\documentclass[letterpaper, 10pt, conference]{ieeeconf}
\IEEEoverridecommandlockouts 
\usepackage{amsmath,amssymb,url}
\usepackage{graphicx,subfigure,warmread}
\usepackage[all,import]{xy}
\usepackage{color}

\newcommand{\norm}[1]{\ensuremath{\left\| #1 \right\|}}

\newcommand{\bracket}[1]{\ensuremath{\left[ #1 \right]}}
\newcommand{\braces}[1]{\ensuremath{\left\{ #1 \right\}}}
\newcommand{\parenth}[1]{\ensuremath{\left( #1 \right)}}

\newcommand{\refeqn}[1]{(\ref{eqn:#1})}
\newcommand{\reffig}[1]{Figure \ref{fig:#1}}
\newcommand{\tr}[1]{\mbox{tr}\ensuremath{\negthickspace\bracket{#1}}}
\newcommand{\trs}[1]{\mathrm{tr}\ensuremath{[#1]}}

\newcommand{\SO}{\ensuremath{\mathsf{SO(3)}}}
\newcommand{\T}{\ensuremath{\mathsf{T}}}
\renewcommand{\L}{\ensuremath{\mathsf{L}}}
\newcommand{\so}{\ensuremath{\mathfrak{so}(3)}}
\newcommand{\SE}{\ensuremath{\mathsf{SE(3)}}}

\renewcommand{\Re}{\ensuremath{\mathbb{R}}}
\newcommand{\Sph}{\ensuremath{\mathsf{S}}}

\newcommand{\D}{\ensuremath{\mathbf{D}}}

\title{\LARGE \bf
Nonlinear Robust Tracking Control of a Quadrotor UAV on \SE}

\author{Taeyoung Lee\authorrefmark{1}, Melvin Leok\authorrefmark{2}, and N. Harris McClamroch%
\thanks{Taeyoung Lee, Mechanical and Aerospace Engineering, The George Washington University, Washington DC 20052 {\tt tylee@gwu.edu}}%
\thanks{Melvin Leok, Mathematics, University of California at San Diego, La Jolla, CA 92093 {\tt mleok@math.ucsd.edu}}%
\thanks{N. Harris McClamroch, Aerospace Engineering, University of Michigan, Ann Arbor, MI 48109 {\tt
nhm@umich.edu}}%
\thanks{\textsuperscript{\footnotesize\ensuremath{*}}This research has been supported in part by NSF under grants CMMI-1029551.}
\thanks{\textsuperscript{\footnotesize\ensuremath{\dagger}}This research has been supported in part by NSF under grants DMS-0726263, DMS-1001521, DMS-1010687, and CMMI-1029445.}
}

\newcommand{\EditTL}[1]{{\color{red}\protect #1}}
\renewcommand{\EditTL}[1]{{\protect #1}}

\newtheorem{prop}{Proposition}

\begin{document}
\allowdisplaybreaks
\maketitle \thispagestyle{empty} \pagestyle{empty}

\begin{abstract}
This paper provides nonlinear tracking control systems for a quadrotor unmanned aerial vehicle (UAV) that are robust to bounded uncertainties. A mathematical model of a quadrotor UAV is defined on the special Euclidean group, and nonlinear output-tracking controllers are developed to follow (1) an attitude command, and (2) a position command for the vehicle center of mass. The controlled system has the desirable properties that the tracking errors are uniformly ultimately bounded, and the size of the ultimate bound can be arbitrarily reduced by control system parameters. Numerical examples illustrating complex maneuvers are provided.
\end{abstract}

\section{INTRODUCTION}

A quadrotor unmanned aerial vehicle (UAV) consists of two pairs of counter-rotating rotors and propellers. It has been envisaged for various applications such as surveillance or mobile sensor networks as well as for educational purposes, and several control systems have been studied.

Linear control systems have been widely used to enhance the stability properties of an equilibrium of a quadrotor UAV~\cite{ValBetPAGNCC06,HofHuaAGNCC07,CasLozICSM05}. In~\cite{GilHofIJRR11}, the quadrotor dynamics is modeled as a collection of simplified hybrid dynamic modes, and  reachability sets are analyzed to guarantees the safety and performance for larger area of operating conditions.

Several nonlinear controllers have been developed as well. Backstepping and sliding mode techniques are applied in~\cite{BouSiePIICRA05,EfePMCCA07}, and a nonlinear $H_\infty$ controller is studied in~\cite{RafOrtA10}. An adaptive neural network based control system is developed in~\cite{NicMacPCCECE08}. Since all of these controllers are based on Euler angles, they exhibit singularities when representing complex rotational maneuvers of a quadrotor UAV, thereby significantly restricting their ability to achieve complex flight maneuvers. 

An attitude control system based on quaternions is applied to a quadrotor UAV~\cite{TayMcGITCSTI06}. Quaternions do not have singularities, but they have ambiguities in representing an attitude, as the three-sphere $\Sph^3$ double-covers $\SO$. As a result, in a quaternion-based attitude control system, convergence to a single attitude implies convergence to either of the two disconnected, antipodal points on $\Sph^3$~\cite{MaySanITAC11}. Therefore, depending on the particular choice of control inputs, a quaternion-based control system may become discontinuous when applied to actual attitude dynamics~\cite{MaySanPACC11b}, and it may also exhibit unwinding behavior, where the controller rotates a rigid body through unnecessarily large angles~\cite{BhaBerSCL00,MaySanPACC11}.

Attitude control systems also have been developed directly on the special orthogonal group, $\SO$ to avoid the singularities associated with Euler-angles and the ambiguity of quaternions~\cite{BulLew05,MaiBerITAC06,CabCunPICDC08,ChaSanICSM11}. By following this geometric approach, the dynamics of a quadrotor UAV is globally expressed on the special Euclidean group, $\SE$, and nonlinear control systems are developed to track outputs of several flight modes, namely an attitude controlled flight mode, a position controlled flight mode, and a velocity controlled flight mode~\cite{LeeLeo}. Several aggressive maneuvers of a quadrotor UAV are also demonstrated based on a hybrid control architecture. This is particularly desirable since complicated reachability set analysis is not required to guarantee a safe switching between different flight modes, as the region of attraction for each flight mode covers the configuration space almost globally.

In this paper, we extend the results of \cite{LeeLeo} to construct nonlinear robust tracking control systems on $\SE$ for a quadrotor UAV. We assume that there exist unstructured, bounded uncertainties, with pre-determined bounds, on the translational dynamics and the rotation dynamics of a quadrotor UAV. Output tracking control systems are developed to follow an attitude command or a position command for the vehicle center of mass. We show that the tracking errors are uniformly ultimately bounded, and the size of the ultimate bound can be arbitrarily reduced. The robustness of the proposed tracking control systems are critical in generating complex maneuvers, as the impact of the several aerodynamic effects resulting from the variation in air speed is significant even at moderate velocities~\cite{HofHuaAGNCC07}.

The paper is organized as follows. We develop a globally defined model for the translational and rotational dynamics of a quadrotor UAV in Section \ref{sec:QDM}. A hybrid control architecture is introduced and a robust attitude tracking control system is developed in Section \ref{sec:ACFM}. Section \ref{sec:PCFM} present results for a robust position tracking, followed by numerical examples in Section \ref{sec:NE}.

\section{QUADROTOR DYNAMICS MODEL}\label{sec:QDM}

Consider a quadrotor UAV model illustrated in \reffig{QM}. This is a system of four identical rotors and propellers located at the vertices of a square, which generate a thrust and torque normal to the plane of this square. We choose an inertial reference frame $\{\vec e_1,\vec e_2,\vec e_3\}$ and a body-fixed frame $\{\vec b_1,\vec b_2,\vec b_3\}$. The origin of the body-fixed frame is located at the center of mass of this vehicle. The first and the second axes of the body-fixed frame, $\vec b_1,\vec b_2$, lie in the plane defined by the centers of the four rotors, as illustrated in \reffig{QM}. The third body-fixed axis $\vec b_3$ is normal to this plane.  Each of the inertial reference frame and the body-fixed reference frame consist of a triad of orthogonal vectors defined according to the right hand rule.    Define
\begin{center}
\begin{tabular}{lp{5.6cm}}
{$m\in\Re$} & the total mass\\
{$J\in\Re^{3\times 3}$} & the inertia matrix with respect to the body-fixed frame\\
{$R\in\SO$} & the rotation matrix from the body-fixed frame to the inertial  frame\\
{$\Omega\in\Re^3$} & the angular velocity in the body-fixed frame\\
{$x\in\Re^3$} & the position vector of the center of mass in the inertial frame\\
{$v\in\Re^3$} & the velocity vector of the center of mass in the inertial frame\\
{$d\in\Re$} & the distance from the center of mass to the center of each rotor in the $\vec b_1,\vec b_2$ plane\\
{$f_i\in\Re$} & the thrust generated by the $i$-th propeller along the $-\vec b_3$ axis\\
{$\tau_i\in\Re$} & the torque generated by the $i$-th propeller about the $\vec b_3$ axis\\
{$f\in\Re$} & the total thrust magnitude, i.e., $f=\sum_{i=1}^4 f_i$\\
{$M\in\Re^3$} & the total moment vector in the body-fixed frame\\
\end{tabular}
\end{center}

The configuration of this quadrotor UAV is defined by the location of the center of mass and the attitude with respect to the inertial frame. Therefore, the configuration manifold is the special Euclidean group $\SE$, which is the semidirect product of $\Re^3$ and the special orthogonal group $\SO=\{R\in\Re^{3\times 3}\,|\, R^TR=I,\, \det{R}=1\}$. 

\begin{figure}
\setlength{\unitlength}{0.65\columnwidth}\footnotesize
\centerline{
\begin{picture}(1,0.8)(0,0)
\put(0,0){\includegraphics[width=0.65\columnwidth]{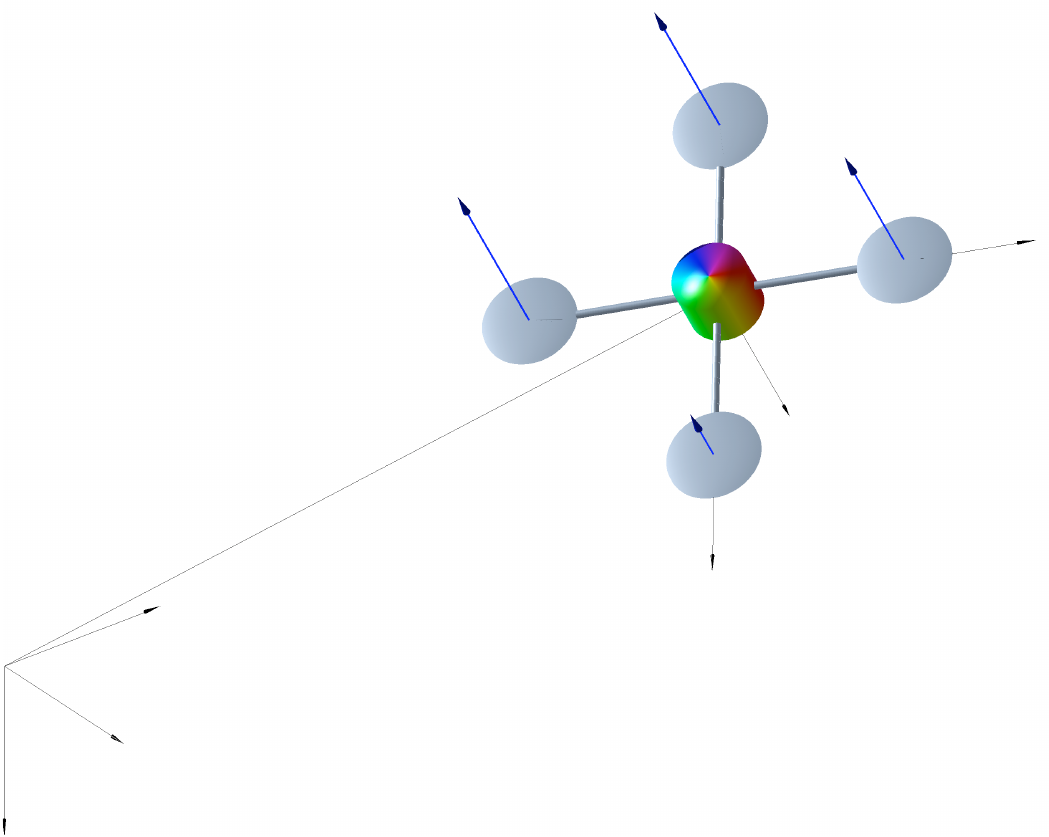}}
\put(0.16,0.18){\shortstack[c]{$\vec e_1$}}
\put(0.13,0.06){\shortstack[c]{$\vec e_2$}}
\put(0.02,0.0){\shortstack[c]{$\vec e_3$}}
\put(0.98,0.5){\shortstack[c]{$\vec b_1$}}
\put(0.70,0.22){\shortstack[c]{$\vec b_2$}}
\put(0.76,0.37){\shortstack[c]{$\vec b_3$}}
\put(0.78,0.66){\shortstack[c]{$f_1$}}
\put(0.56,0.76){\shortstack[c]{$f_2$}}
\put(0.40,0.63){\shortstack[c]{$f_3$}}
\put(0.61,0.42){\shortstack[c]{$f_4$}}
\put(0.30,0.35){\shortstack[c]{$x$}}
\put(0.90,0.35){\shortstack[c]{$R$}}
\end{picture}}
\caption{Quadrotor model}\label{fig:QM}
\end{figure}
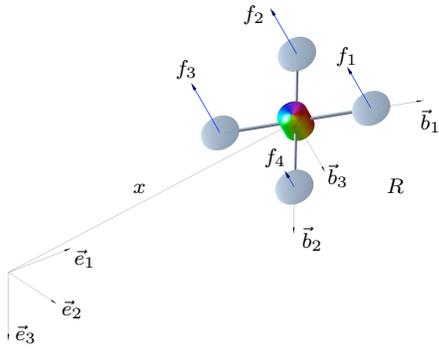

The following conventions are assumed for the rotors and propellers, and the thrust and moment that they exert on the quadrotor UAV. We assume that the thrust of each propeller is directly controlled, and the direction of the thrust of each propeller is normal to the quadrotor plane. The first and third propellers are assumed to generate a thrust along the direction of $-\vec b_3$ when rotating clockwise; the second and fourth propellers are assumed to generate a thrust along the same direction of $-\vec b_3$ when rotating counterclockwise. Thus, the thrust magnitude is $f=\sum_{i=1}^4 f_i$, and it is positive when the total thrust vector acts  along $-\vec b_3$, and it is negative when the total thrust vector acts along $\vec b_3$. By the definition of the rotation matrix $R\in\SO$, 
the total thrust vector is given by $-fRe_3\in\Re^3$ in the inertial frame.
We also assume that the torque generated by each propeller is directly proportional to its thrust. Since it is assumed that the first and the third propellers rotate clockwise and the second and the fourth propellers rotate counterclockwise to generate a positive thrust along the direction of $-\vec b_3$, %
the torque generated by the $i$-th propeller about $\vec b_3$ can be written as $\tau_i=(-1)^{i} c_{\tau f} f_i$  for a fixed constant $c_{\tau f}$.    All of these assumptions are common~\cite{CasLozICSM05,TayMcGITCSTI06}.  

Under these assumptions, the moment vector in the body-fixed frame is given by 
\begin{align}
\begin{bmatrix} f \\ M_1 \\ M_2 \\ M_3 $ $\end{bmatrix}
=\begin{bmatrix} 1 & 1& 1& 1\\
0 & -d & 0 & d\\
d & 0 & -d & 0\\
-c_{\tau f} & c_{\tau f} & -c_{\tau f} &c_{\tau f}
\end{bmatrix}
\begin{bmatrix} f_1 \\ f_2\\ f_3 \\ f_4\end{bmatrix}\label{eqn:fM}.
\end{align}
The determinant of the above $4\times4$ matrix is $8 c_{\tau f} d^2$, so it is invertible when $d\neq 0$ and $c_{\tau f}\neq 0$. Therefore, for given thrust magnitude $f$ and given moment vector $M$, the thrust of each propeller $f_1, f_2, f_3, f_4$ can be obtained from \refeqn{fM}. Using this equation, the thrust magnitude $f\in\Re$ and the moment vector $M\in\Re^3$ are viewed as control inputs in this paper.

The equations of motion of the quadrotor UAV can be written as
\begin{gather}
\dot x  = v,\label{eqn:EL1}\\
m \dot v = mge_3 - f R e_3 + \Delta_x,\label{eqn:EL2}\\
\dot R = R\hat\Omega,\label{eqn:EL3}\\
J\dot \Omega + \Omega\times J\Omega = M + \Delta_R,\label{eqn:EL4}
\end{gather}
where the \textit{hat map} $\hat\cdot:\Re^3\rightarrow\so$ is defined by the condition that $\hat x y=x\times y$ for all $x,y\in\Re^3$ (see Appendix \ref{app:hat}). The inverse of the hat map is denoted by the \textit{vee} map, $\vee:\so\rightarrow\Re^3$. Unstructured uncertainties in the translational dynamics and the rotational dynamics of a quadrotor UAV are denoted by $\Delta_x$ and $\Delta_R\in\Re^3$, respectively. We assume that uncertainties are bounded:
\begin{align}
\|\Delta_x\|\leq \delta_x, \quad \|\Delta_R\|\leq \delta_R
\end{align}
for known, positive constants $\delta_x$, $\delta_R\in\Re$.

Throughout this paper, $\lambda_m(\cdot)$ and $\lambda_{M}(\cdot)$ denote the minimum eignevalue and the maximum eigenvalue of a matrix, respectively.

\section{ATTITUDE CONTROLLED FLIGHT MODE}\label{sec:ACFM}

\subsection{Flight Modes}

Since the quadrotor UAV has four inputs, it is possible to achieve asymptotic output tracking for at most four quadrotor UAV outputs.    The quadrotor UAV has three translational and three rotational degrees of freedom; it is not possible to achieve asymptotic output tracking of both attitude and position of the quadrotor UAV.    This motivates us to introduce several flight modes, namely (1) an attitude controlled flight mode, and (2) a position controlled flight mode. 

A complex flight maneuver can be defined by specifying a concatenation of flight modes together with conditions for switching between them; for each flight mode one also specifies the desired or commanded outputs as functions of time. Unlike a hybrid flight control system that requires reachability analyses~\cite{GilHofIJRR11}, the proposed control system is robust to switching conditions since each flight mode has almost global stability properties, and it is straightforward to design a complex maneuver of a quadrotor UAV. 

In this section, an attitude controlled flight mode is considered, where the outputs are the attitude of the quadrotor UAV and the controller for this flight mode achieves asymptotic attitude tracking.

\subsection{Attitude Tracking Errors}

Suppose that an arbitrary smooth attitude command $R_d(t)\in\SO$ is given. The corresponding angular velocity command is obtained by the attitude kinematics equation, $\hat\Omega_d = R_d^T \dot R_d$. 

We first define errors associated with the attitude dynamics of the quadrotor UAV. The attitude error function studied in~\cite{BulLew05,ChaMcCITAC09,Lee11}, and several properties are summarized as follows.

\begin{prop}\label{prop:1}
For a given tracking command $(R_d,\Omega_d)$, and the current attitude and angular velocity $(R,\Omega)$, we define an attitude error function $\Psi:\SO\times\SO\rightarrow\Re$, an attitude error vector $e_R\in\Re^3$, and an angular velocity error vector $e_\Omega\in \Re^3$ as follows:
\begin{gather}
\Psi (R,R_d) = \frac{1}{2}\tr{I-R_d^TR},\\
e_R =\frac{1}{2} (R_d^TR-R^TR_d)^\vee,\\
e_\Omega = \Omega - R^T R_d\Omega_d,
\end{gather}
Then, the following statements hold:
\begin{itemize}
\item[(i)] $\Psi$ is locally positive-definite about $R=R_d$.
\item[(ii)] the left-trivialized derivative of $\Psi$ is given by
\begin{align}
\T^*_I \L_R\, (\D_R\Psi(R,R_d))= e_R.
\end{align}
\item[(iii)] the critical points of $\Psi$, where $e_R=0$, are $\{R_d\}\cup\{R_d\exp (\pi \hat s),\,s\in\Sph^2 \}$.
\item[(iv)] a lower bound of $\Psi$ is given as follows:
\begin{align}
\frac{1}{2}\|e_R(R,R_d)\|^2 \leq \Psi(R,R_d),\label{eqn:PsiLB}
\end{align}
\item[(v)] Let $\psi$ be a positive constant that is strictly less than $2$. If $\Psi(R,R_d)< \psi<2$, then an upper bound of $\Psi$ is given by
\begin{align}
\Psi(R,R_d)\leq \frac{1}{2-\psi} \|e_R(R,R_d)\|^2.\label{eqn:PsiUB}
\end{align}
\end{itemize}
\end{prop}
\begin{proof} See \cite{Lee11}. \end{proof}

\subsection{Attitude Tracking Controller}

We now introduce a nonlinear controller for the attitude controlled flight mode, described by an expression for the moment vector:
\begin{align}
M & = -k_R e_R -k_\Omega e_\Omega +\Omega\times J\Omega\nonumber\\
&\qquad -J(\hat\Omega R^T R_d \Omega_d - R^T R_d\dot\Omega_d)+\mu_R,\label{eqn:aM}\\
\mu_R & = -\frac{\delta_R^2 e_A}{\delta_R \|e_A\|+\epsilon_R},\label{eqn:muR}\\
e_A & = e_\Omega + c_2 J^{-1} e_R,\label{eqn:eA}
\end{align}
where $k_R,k_\Omega,c_2,\epsilon_R$ are positive constants.

In this attitude controlled mode, it is possible to ignore the translational motion of the quadrotor UAV; consequently the reduced model for the attitude dynamics are given by equations \refeqn{EL3}, \refeqn{EL4}, using the controller expression \refeqn{aM}-\refeqn{eA}.  We now state the result that the tracking errors $(e_R, e_\Omega)$ are uniformly ultimately bounded.

\begin{prop}{(Robustness of Attitude Controlled Flight Mode)}\label{prop:Att}
Suppose that the initial attitude error satisfies
\begin{gather}
\Psi(R(0),R_d(0))<\psi_2<2\label{eqn:eRb0}
\end{gather}
for a constant $\psi_2$. Consider the control moment $M$ defined in \refeqn{aM}-\refeqn{eA}. For positive constants $k_R,k_\Omega$, the constants $c_2,\epsilon_R$ are chosen such that
\begin{gather}
c_2 < \min\bigg\{  k_\Omega,\frac{4k_\Omega k_R\lambda_{m}(J)^2}{k_\Omega^2\lambda_{M}(J)+4k_R\lambda_{m}(J)^2},\sqrt{k_R\lambda_{m}(J)}\bigg\},\label{eqn:c2}\\
\epsilon_R < \frac{\lambda_m(M_{21})\lambda_m(W_2)}{\lambda_M(M_{22})}\psi_2(2-\psi_2),\label{eqn:epsilonR}
\end{gather}
where the matrices $M_{21},M_{22},W_2\in\Re^{2\times2}$ are given by
\begin{gather*}
M_{21} = \frac{1}{2}\begin{bmatrix} k_R & -c_2 \\ -c_2 & \lambda_{m}(J)  \end{bmatrix},\,
M_{22} = \frac{1}{2}\begin{bmatrix} \frac{2k_R}{2-\psi_2} & c_2 \\ c_2 & \lambda_{M}(J)\end{bmatrix},
\\
W_2 = \begin{bmatrix} \frac{c_2k_R}{\lambda_{M}(J)} & -\frac{c_2k_\Omega}{2\lambda_{m}(J)} \\ 
-\frac{c_2k_\Omega}{2\lambda_{m}(J)} & k_\Omega-c_2 \end{bmatrix}.
\end{gather*}

Then, the attitude tracking errors $(e_R,e_\Omega)$ are uniformly ultimately bounded, and the ultimate bound is given by
\begin{align}
\braces{\|e_R\|^2 + \|e_\Omega\|^2 \leq \frac{\lambda_M(M_{22})}{\lambda_m(M_{21})\lambda_m(W_2)}\epsilon_R}.\label{eqn:uubA}
\end{align}
\end{prop}

\begin{proof}
See Appendix \ref{sec:pfAtt}.
\end{proof}

From \refeqn{eRb0}, the initial attitude error should be less than $180^\circ$, in terms of the rotation angle about the eigenaxis between $R$ and $R_d$. We can further show that the attitude tracking errors exponentially converges to \refeqn{uubA}, where the size of the ultimate bound can be reduced by the controller parameter $\epsilon_R$. It is also possible to achieve exponential attractiveness if the constant $\epsilon_R$ in \refeqn{muR} is replaced by $\epsilon_R \exp(-\beta t)$ for $\beta >0$. All of these results can be applied to a nonlinear robust control problem for the attitude dynamics of any rigid body.

Asymptotic tracking of the quadrotor attitude does not require specification of the thrust magnitude.   As an auxiliary problem, the thrust magnitude can be chosen in many different ways to achieve an additional translational motion objective. For example, it can be used to asymptotically track a quadrotor altitude command~\cite{LeeLeo}.

Since the translational motion of the quadrotor UAV can only be partially controlled; this flight mode is most suitable for short time periods where an attitude maneuver is to be completed.   


\section{POSITION CONTROLLED FLIGHT MODE}\label{sec:PCFM}

We now introduce a nonlinear controller for the position controlled flight mode. This flight mode requires analysis of the coupled translational and rotational equations of motion; hence, we make use of the notation and analysis in the prior section to describe the properties of the closed loop system in this flight mode.  

\subsection{Position Tracking Errors}

An arbitrary smooth position tracking command $x_d(t) \in \Re^3$ is chosen.    The position tracking errors for the position and the velocity are given by:
\begin{align}
e_x & = x - x_d,\\
e_v & = v - \dot x_d.
\end{align}
Following the prior definition of the attitude error and the angular velocity error, we define
\begin{align}
e_R = \frac{1}{2} (R_c^TR - R^T R_c)^\vee, \quad 
e_\Omega = \Omega - R^T R_c \Omega_c\label{eqn:eWc},
\end{align}
and the computed attitude $R_c(t)\in\SO$ and computed angular velocity $\Omega_c \in \mathbb{R}^3$ are given by
\begin{align}
R_c=[ b_{1_c};\, b_{3_c}\times b_{1_c};\, b_{3_c}],\quad \hat\Omega_c = R_c^T \dot R_c\label{eqn:RdWc},
\end{align}
where $ b_{3_c} \in \Sph^2$ is defined by
\begin{align}
 b_{3_c} = -\frac{-k_x e_x - k_v e_v - mg e_3 +m\ddot x_d+\mu_x}{\norm{-k_x e_x - k_v e_v - mg e_3 + m\ddot x_d+\mu_x}},\label{eqn:Rd3}
\end{align}
and $ b_{1c} \in \Sph^2$ is selected to be orthogonal to $ b_{3c}$, thereby guaranteeing that $R_c \in \SO$. The constants $k_x,k_v$ are positive, and the control input term $\mu_x$ is defined later in \refeqn{mux}. We assume that 
\begin{align}
\norm{-k_x e_x - k_v e_v - mg e_3 + m\ddot x_d+\mu_x} \neq 0,\label{eqn:A1}
\end{align}
and the commanded acceleration is uniformly bounded such that
\begin{align}
\|-mge_3+m\ddot x_d\| < B\label{eqn:B}
\end{align}
for a given positive constant $B$. 

\subsection{Position Tracking Controller}

The nonlinear controller for the position controlled flight mode, described by control expressions for the  thrust magnitude and the  moment vector, are:
\begin{align}
f & = ( k_x e_x + k_v e_v + mg e_3-m\ddot x_d-\mu_x)\cdot Re_3,\label{eqn:f}\\
M & = -k_R e_R -k_\Omega e_\Omega +\Omega\times J\Omega\nonumber\\
&\quad -J(\hat\Omega R^T R_c \Omega_c - R^T R_c\dot\Omega_c+\mu_R),\label{eqn:M}\\
\mu_x & = -\frac{\delta_x^{\tau+2} e_B \|e_B\|^\tau}{\delta_x^{\tau+1}\|e_B\|^{\tau+1}+\epsilon_x^{\tau+1}},\label{eqn:mux}\\
e_B & = e_v + \frac{c_1}{m} e_x,\label{eqn:eB}\\
\mu_R & = -\frac{\delta_R^2 e_A}{\delta_R \|e_A\|+\epsilon_R},\label{eqn:muR1}\\
e_A & = e_\Omega + c_2 J^{-1} e_R,\label{eqn:eA1}
\end{align}
where $k_x,k_v,k_R,k_\Omega,c_1,c_2,\epsilon_x,\epsilon_R,\tau$ are positive constants, and $\tau > 2$.


The nonlinear controller given by equations \refeqn{f}, \refeqn{M} can be given a backstepping interpretation.   The computed attitude $R_c$ given in equation \refeqn{RdWc} is selected so that the thrust axis $-b_3$ of the quadrotor UAV tracks the computed direction given by $-b_{3_c}$ in \refeqn{Rd3}, which is a direction of the thrust vector that achieves position tracking.   The moment expression \refeqn{M} causes the attitude of the quadrotor UAV to asymptotically track $R_c$ and the thrust magnitude expression \refeqn{f} achieves asymptotic position tracking.  

The closed loop system for this position controlled flight mode is illustrated in \reffig{CS}.   The corresponding closed loop control system is described by equations \refeqn{EL1}-\refeqn{EL4}, using the controller expressions \refeqn{f}-\refeqn{eA1}. We now state the result that the tracking errors $(e_x, e_v, e_R, e_\Omega)$ are uniformly ultimately bounded.

\begin{figure}
\centerline{
\setlength{\unitlength}{2.1em}\centering\footnotesize
\begin{picture}(11,3.2)(0.0,-3.2)
\put(1.2,-1.5){\framebox(2,1.0)[c]
{\shortstack[c]{Force\\controller}}}
\put(4.0,-2.2){\framebox(2,1.0)[c]
{\shortstack[c]{Moment\\controller}}}
\put(0,-0.95){\vector(1,0){1.2}}
\put(3.2,-1.35){\vector(1,0){0.8}}
\put(3.2,-0.65){\vector(1,0){4.1}}
\put(0,-1.8){\vector(1,0){4.0}}
\put(6,-1.7){\vector(1,0){1.3}}
\put(7.3,-2.25){\framebox(2.5,1.8)[c]
{\shortstack[c]{Quadrotor\\Dynamics}}}
\put(9.8,-1.35){\vector(1,0){1.2}}
\put(6.6,-1.0){$f$}
\put(6.5,-2.05){$M$}
\put(3.30,-1.15){$b_{3_c}$}
\put(0.25,-0.80){$x_d$}
\put(0.0,-1.65){($b_{1_d}$)}
\put(5.5,-3.2){$x,v,R,\Omega$}
\put(10.3,-1.35){\line(0,-1){1.5}}
\put(10.3,-2.85){\line(-1,0){8.1}}
\put(2.2,-2.85){\vector(0,1){1.35}}
\put(2.2,-2.10){\vector(1,0){1.8}}
\put(2.2,-2.10){\circle*{0.1}}
\put(10.3,-1.35){\circle*{0.1}}
\put(0.9,-2.4){\dashbox{0.08}(5.4,2.1)[c]{}}
\put(2.8,-2.7){Controller}
\end{picture}
}
\caption{Controller structure for position controlled flight mode}\label{fig:CS}
\end{figure}
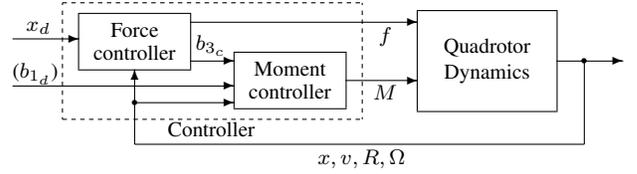

\begin{prop}{(Robustness of Position Controlled Flight Mode)}\label{prop:Pos}
Suppose that the initial conditions satisfy
\begin{gather}
\Psi(R(0),R_c(0)) < \psi_1 < 1,\label{eqn:Psi0}\\
\|e_x(0)\| < e_{x_{\max}},\label{eqn:ex0}
\end{gather}
for positive constants $\psi_1, e_{x_{\max}}$. Consider the control inputs $f,M$ defined in \refeqn{f}-\refeqn{eA1}.
For positive constants $k_x,k_v$, we choose positive constants $c_1,c_2,k_R,k_\Omega,\epsilon_x,\epsilon_R$ such that
\begin{gather}
c_1 < \min\braces{k_v(1-\alpha),\;\frac{4mk_xk_v(1-\alpha)^2}{k_v^2(1+\alpha)^2+4m k_x(1-\alpha)},\; \sqrt{k_xm} },\label{eqn:c1b}\\
c_2 < \min\bigg\{ k_\Omega, \frac{4k_\Omega k_R\lambda_{m}(J)^2}{k_\Omega^2\lambda_{M}(J)+4k_R\lambda_{m}(J)^2}, \sqrt{k_R\lambda_{m}(J)}\bigg\},\label{eqn:c2b}\\
\lambda_{m}(W_2) > \frac{\|W_{12}\|^2}{4\lambda_{m}(W_1)},\label{eqn:kRkWb}\\
\begin{aligned}
&\epsilon_x+\epsilon_R <\nonumber\\
& \frac{\min\{ \lambda_m(M_{11}), \lambda_m(M_{21})\}
\min\{ e_{x_{\max}}^2,\psi_1(2-\psi_1)\}}%
{\max\{ \lambda_M(M_{12}),\lambda_M(M_{22}')\}}  \lambda_m(W),
\end{aligned}\\\label{eqn:epsilon_bound}
\end{gather}
where $\alpha=\sqrt{\psi_1(2-\psi_1)}$, and the matrices $M_{11},M_{12},M_{21}$, $M_{22}',W_1,W_{12},W_2,W\in\Re^{2\times 2}$ are given by
\begin{gather*}
M_{11} = \frac{1}{2}\begin{bmatrix} k_x & -c_1 \\ -c_1 & m\end{bmatrix},\quad
M_{12} = \frac{1}{2}\begin{bmatrix} k_x & c_1 \\ c_1 & m\end{bmatrix},\\
M_{21} = \frac{1}{2}\begin{bmatrix} k_R & -c_2 \\ -c_2 & \lambda_{m}(J)  \end{bmatrix},\quad
M'_{22} = \frac{1}{2}\begin{bmatrix} \frac{2k_R}{2-\psi_1} & c_2 \\ c_2 & \lambda_{M}(J)\end{bmatrix},
\end{gather*}
\begin{align*}
W_1 &= \begin{bmatrix} \frac{c_1k_x}{m}(1-\alpha) & -\frac{c_1k_v}{2m}(1+\alpha)\\
-\frac{c_1k_v}{2m}(1+\alpha) & k_v(1-\alpha)-c_1\end{bmatrix},\\
W_{12}&=\begin{bmatrix}
\frac{c_1}{m}(B+\delta_x) & 0 \\ B+\delta_x+k_xe_{x_{\max}} & 0\end{bmatrix},\\
W_2 &= \begin{bmatrix} \frac{c_2k_R}{\lambda_{M}(J)} & -\frac{c_2k_\Omega}{2\lambda_{m}(J)} \\ 
-\frac{c_2k_\Omega}{2\lambda_{m}(J)} & k_\Omega-c_2 \end{bmatrix},\\
W & =\begin{bmatrix}
\lambda_{m}(W_1) & -\frac{1}{2}\|W_{12}\|_2\\
-\frac{1}{2}\|W_{12}\|_2 & \lambda_m(W_2)
\end{bmatrix}.
\end{align*}
Then, the tracking errors $(e_x,e_v,e_R,e_\Omega)$ are uniformly ultimately bounded, and the ultimate bound is given by
\begin{align}
\Big\{ \|e_x\|^2&+\|e_v\|^2+\|e_R\|^2+\|e_\Omega\|^2 < \nonumber\\
&\frac{\max\{ \lambda_M(M_{12}),\lambda_M(M_{22}')\}}%
{\min\{ \lambda_m(M_{11}), \lambda_m(M_{21})\}\lambda_m(W)}(\epsilon_x+\epsilon_R)
\Big\} 
\label{eqn:uubPos}.
\end{align} 
\end{prop}

\begin{proof}
See Appendix \ref{sec:pfPos}.
\end{proof}

This proposition shows that the proposed control system is robust to bounded, and unstructured uncertainties in the dynamics of a quadrotor UAV. Similar to Proposition \ref{prop:Att}, the ultimate bound can be arbitrarily reduced by choosing smaller $\epsilon_x,\epsilon_R$, and it is possible to obtain exponential attractiveness.

Proposition \ref{prop:Pos} requires that the initial attitude error is less than $90^\circ$ in \refeqn{Psi0}. Suppose that this is not satisfied, i.e. $1\leq\Psi(R(0),R_c(0))<2$.    We can still apply Proposition \ref{prop:Att}, which states that the attitude error exponentially decreases until it enters the ultimate bound given by \refeqn{uubA}. If the constant $\epsilon_R$ is sufficiently small, we can guarantee that the attitude error function decreases to satisfy \refeqn{Psi0} in a finite time. Therefore, by combining the results of Proposition \ref{prop:Att} and \ref{prop:Pos}, we can show ultimate boundedness of the tracking errors when $\Psi(R(0),R_c(0))<2$.

\begin{prop}{(Robustness of Position Controlled Flight Mode with a Larger Initial Attitude Error)}\label{prop:Pos2}
Suppose that the initial conditions satisfy
\begin{gather}
1\leq \Psi(R(0),R_c(0))<\psi_2 < 2\label{eqn:eRb3},\\
\|e_x(0)\| < e_{x_{\max}},
\end{gather}
for a constant $\psi_2,e_{x_{\max}}$. Consider the control inputs $f,M$ defined in \refeqn{f}-\refeqn{eA1}, where the control parameters $k_x,k_v,k_R,k_\Omega,c_1,c_2,\epsilon_x,\epsilon_R$ satisfy \refeqn{c1b}-\refeqn{epsilon_bound} for a positive constant $\psi_1<1$. If the constant $\epsilon_R$ is sufficiently small such that
\begin{align}
\epsilon_R < \frac{\lambda_m(M_{21})\lambda_m(W_2)}{\lambda_M(M_{22})} 
\psi_1(2-\psi_1),\label{eqn:epsilonR_bound_Pos2}
\end{align}
then the tracking errors $(e_x,e_v,e_R,e_\Omega)$ are uniformly ultimately bounded.
\end{prop}

\begin{proof}
See Appendix \ref{sec:pfPos2}.
\end{proof}

\subsection{Direction of the First Body-Fixed Axis}

As described above, the construction of the orthogonal matrix $R_c$ involves having its third column $b_{3_c}$ specified by a normalized feedback function, and its first column $b_{1_c}$ is chosen to be orthogonal to the third column. The unit vector $b_{1_c}$ can be arbitrarily chosen in the plane normal to $b_{3_c}$, which corresponds to a one-dimensional degree of choice. This reflects the fact that the quadrotor UAV has four control inputs that are used to track a three-dimensional position command.

By choosing $b_{1_c}$ properly, we constrain the asymptotic direction of the first body-fixed axis. 
Here, we propose to specify the \textit{projection} of the first body-fixed axis onto the plane normal to $b_{3_c}$. In particular, we choose a desired direction $b_{1_d}\in\Sph^2$, that is not parallel to $b_{3_c}$, and $b_{1_c}$ is selected as $b_{1_c}=\mathrm{Proj}[b_{1_d}]$, where $\mathrm{Proj}[\cdot]$ denotes the normalized projection onto the plane perpendicular to $b_{3_c}$. In this case, the first body-fixed axis does not converge to $b_{1_d}$, but it converges to the projection of $b_{1_d}$, i.e. $b_1\rightarrow b_{1_c}=\mathrm{Proj}[b_{1_d}]$ as $t\rightarrow\infty$, up to the ultimate bound described by \refeqn{uubPos}. In other words, the first body-fixed axis converges to a small neighborhood of the intersection of the plane normal to $b_{3_c}$ and the plane spanned by $b_{3_c}$ and $b_{1_d}$. This can be used to specify the heading direction of a quadrotor UAV in the horizontal plane (see \reffig{b1d} and \cite{LeeLeo} for details).



\begin{figure}
\centerline{
\renewcommand{\xyWARMinclude}[1]{\includegraphics[width=0.42\columnwidth]{#1}}
{\footnotesize\selectfont
$$\begin{xy}
\xyWARMprocessEPS{b1d}{pdf}
\xyMarkedImport{}
\xyMarkedMathPoints{1-15}
\end{xy}
$$}}
\caption{Convergence property of the first body-fixed axis: in the proposed control system, $b_{3_c}$ is determined by \refeqn{Rd3}. We choose a desired direction of the first body fixed axis, namely $b_{1_d}$ that is not parallel to $b_{3_c}$, and project it on to the plane normal to $b_{3_c}$ to obtain $b_{1_c}$. This guarantees that the first body-fixed axis converges to $b_{1_c}$, and therefore it asymptotically lies in the plane spanned by $b_{1_d}$ and $b_{3_c}$. As $b_{3_c}$ converges to the direction of $ge_3-\ddot x_d$ in \refeqn{Rd3}, this allows us to specify the direction of the first body-fixed axis in the plane normal to $ge_3-\ddot x_d$. For all cases, the ultimate convergence error is described by \refeqn{uubPos}.}\label{fig:b1d}
\vspace*{-0.4cm}
\end{figure}
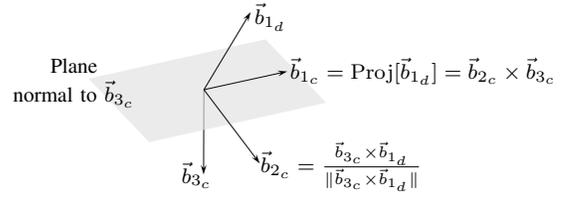

\section{NUMERICAL EXAMPLES}\label{sec:NE}
Numerical results are presented to demonstrate the prior approach for performing complex flight maneuvers. The parameters are chosen to match a quadrotor UAV described in~\cite{PouMahACRA06}.
\begin{gather*}
J=[0.0820,0.0845,0.1377]\,\mathrm{kg-m^2},\quad m=4.34\,\mathrm{kg}\\
 d=0.315\,\mathrm{m},\quad c_{\tau f}=8.004\times 10^{-3}\,\mathrm{m}.
\end{gather*}
The controller parameters are chosen as follows:
\begin{gather*}
k_x=59.02,\quad k_v=24.30,\quad k_R=8.81,\quad k_\Omega = 1.54\\
c_1=3.6,\quad c_2=0.6,\quad \epsilon_x=\epsilon_R=0.04.
\end{gather*}
We consider a fixed disturbance for the translational dynamics, and an oscillatory disturbance for the rotational dynamics as follows:
\begin{align*}
\Delta_x &= [2.50,\,
    1.25,\,
    2.00]^T\,\mathrm{N},\\
\Delta_R(t) & = \frac{2}{\sqrt{3}}[\sin(8\pi t),\, \sin(\pi t),\, \cos(4\pi t)]^T\,\mathrm{Nm}.
\end{align*}
The corresponding bounds of the disturbances are given by $\delta_x = 4.34 $ and $\delta_R=2$. We consider the following two cases.

\paragraph*{Case I (elliptic helix)}
The initial conditions are given by
\begin{gather*}
x(0)=[0.1,0,0]^T\,\mathrm{m},\quad v(0)=[0,0,0]^T\,\mathrm{m/s},\\
R(0)=I,\quad \Omega(0)=[0,0,0]^T\,\mathrm{rad/s}.
\end{gather*}
The desired position command is an elliptic helix, given by
\begin{align*}
x_d(t)=[0.4t,\, 0.4\sin(\pi t),\, -0.6\cos(\pi t)]^T\,\mathrm{m},
\end{align*}
and the desired heading direction is fixed as $\vec b_{1_d}=[1,0,0]^T$. This corresponds to the position controlled flight mode described in Proposition \ref{prop:Pos}, as the initial attitude error is $\Psi(0)=0.14<1$. 

Figure \ref{fig:I} shows simulation results, where the position tracking error converges to a small neighborhood of the zero tracking errors, and the terminal tracking error is $1.2\,\mathrm{cm}$. For comparison, we set the robust control input terms to zero, i.e. $\mu_x=\mu_R=0$, and we repeat numerical simulations to obtain Figure \ref{fig:Ib}. It is observed that the angular velocity tracking error is mostly driven by the disturbance $\Delta_R$, and the corresponding position tracking error is larger than $0.1\,\mathrm{m}$. This illustrates the robustness of the proposed control system for a complex maneuver with larger disturbances.

\begin{figure}
\centerline{
	\subfigure[Attitude error function $\Psi$]{
		\includegraphics[width=0.48\columnwidth]{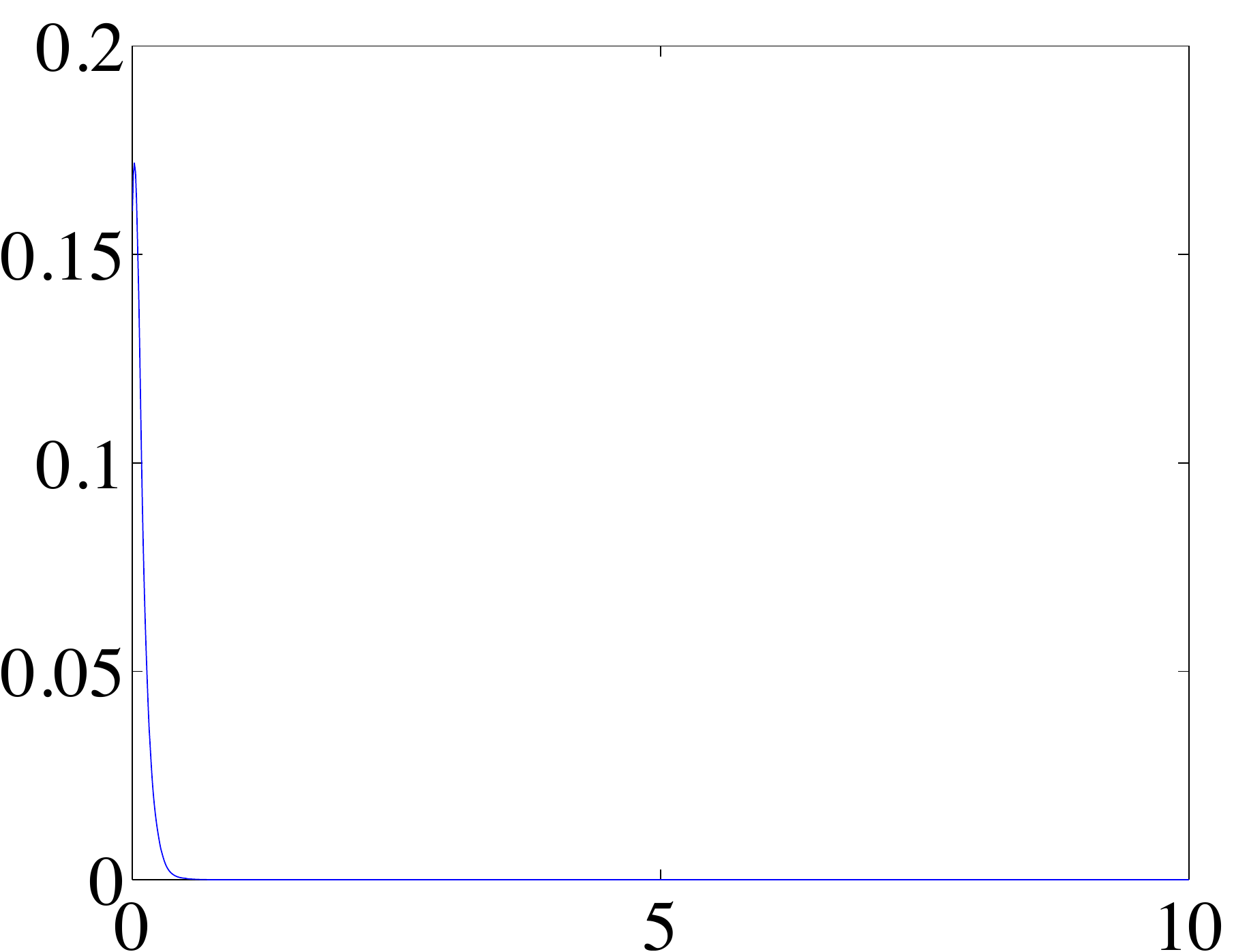}\label{fig:IPsi}}
		\hfill
	\subfigure[Position error $e_x$ ($\mathrm{m}$)]{
		\includegraphics[width=0.485\columnwidth]{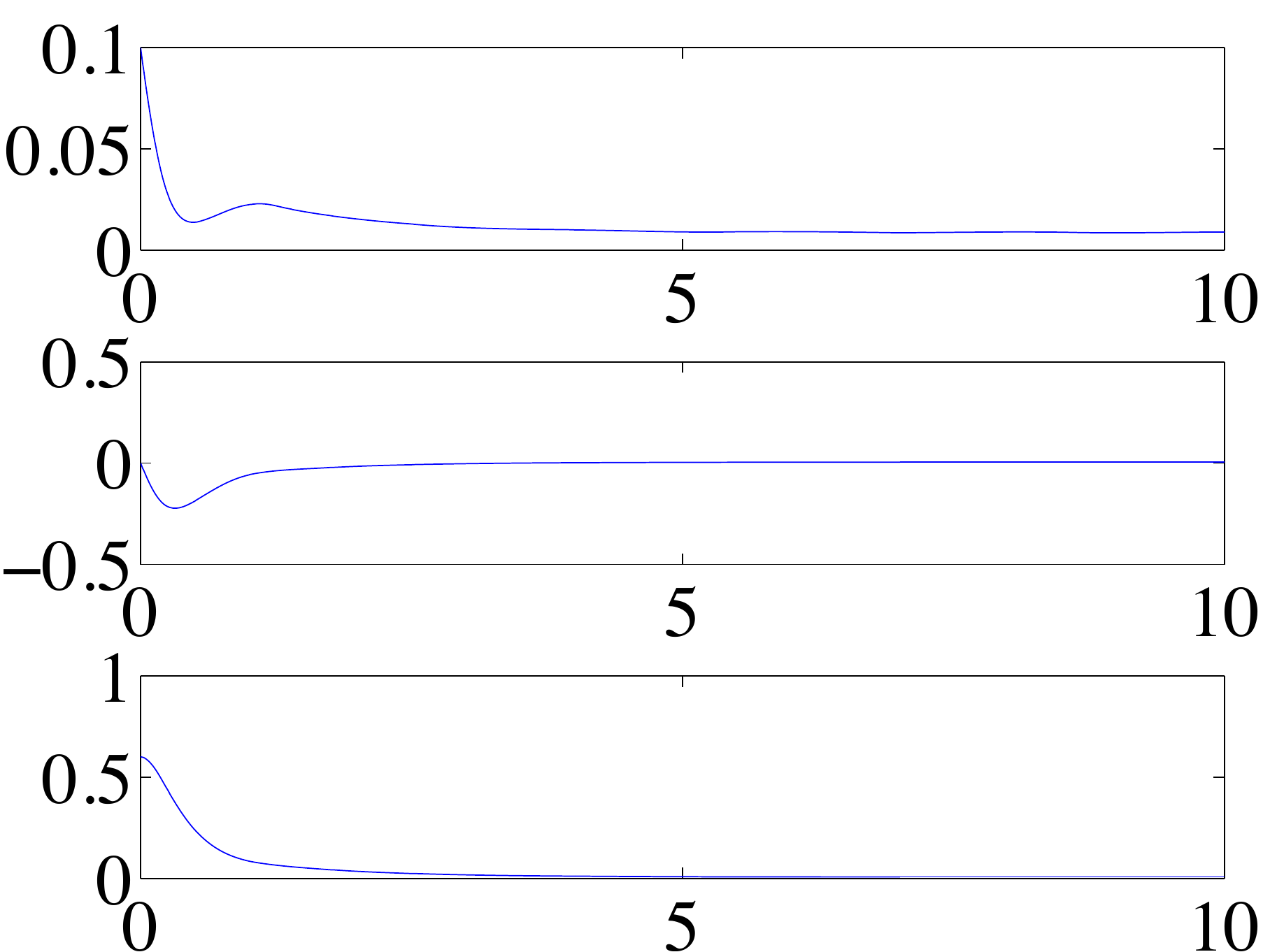}\label{fig:Ix}}
}
\centerline{
	\subfigure[Angular velocity error $e_\Omega$ ($\mathrm{rad/sec}$)]{
		\includegraphics[width=0.48\columnwidth]{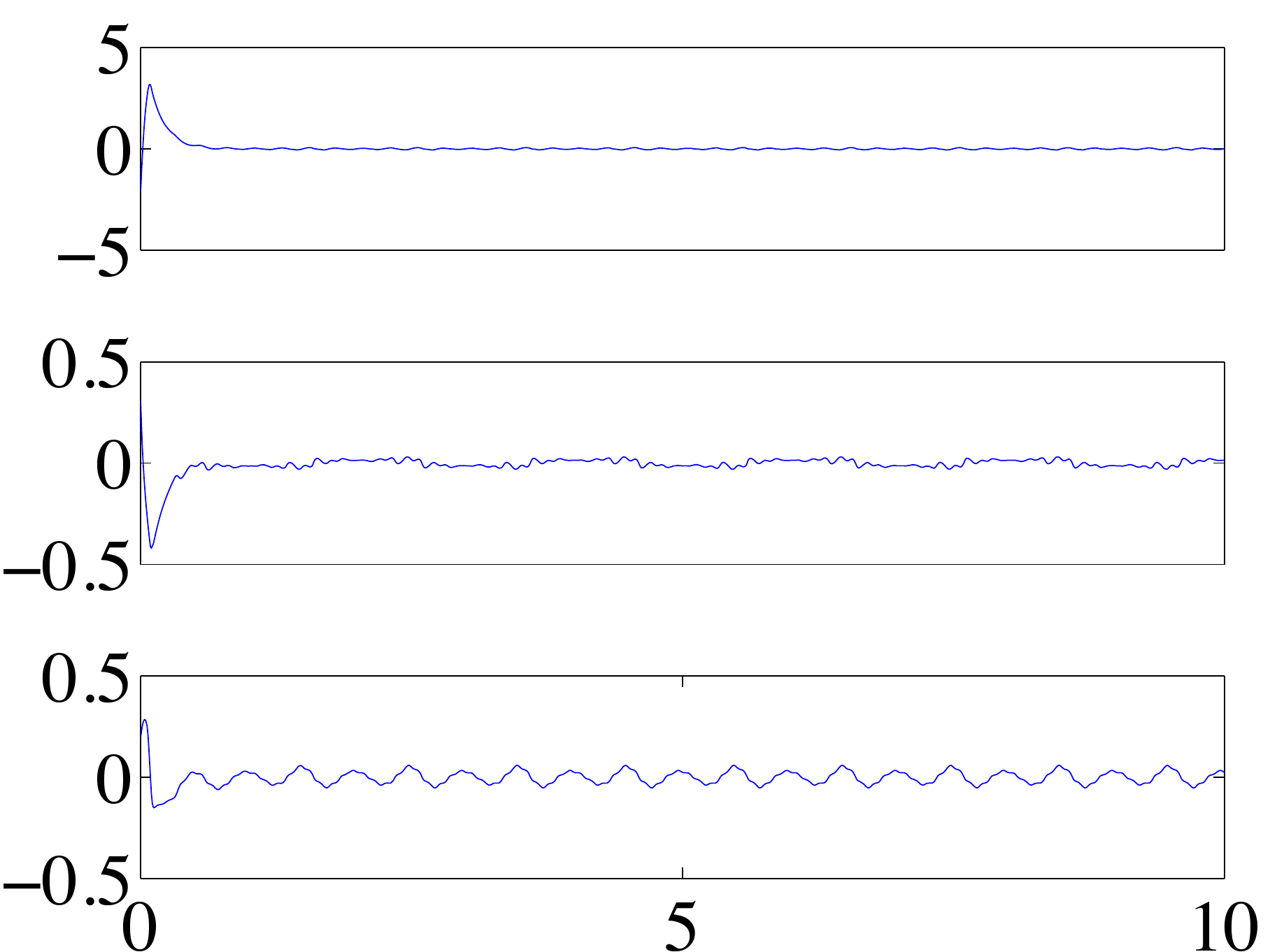}\label{fig:IW}}
		\hfill
	\subfigure[Thrust of each rotor ($\mathrm{N}$)]{
		\includegraphics[width=0.50\columnwidth]{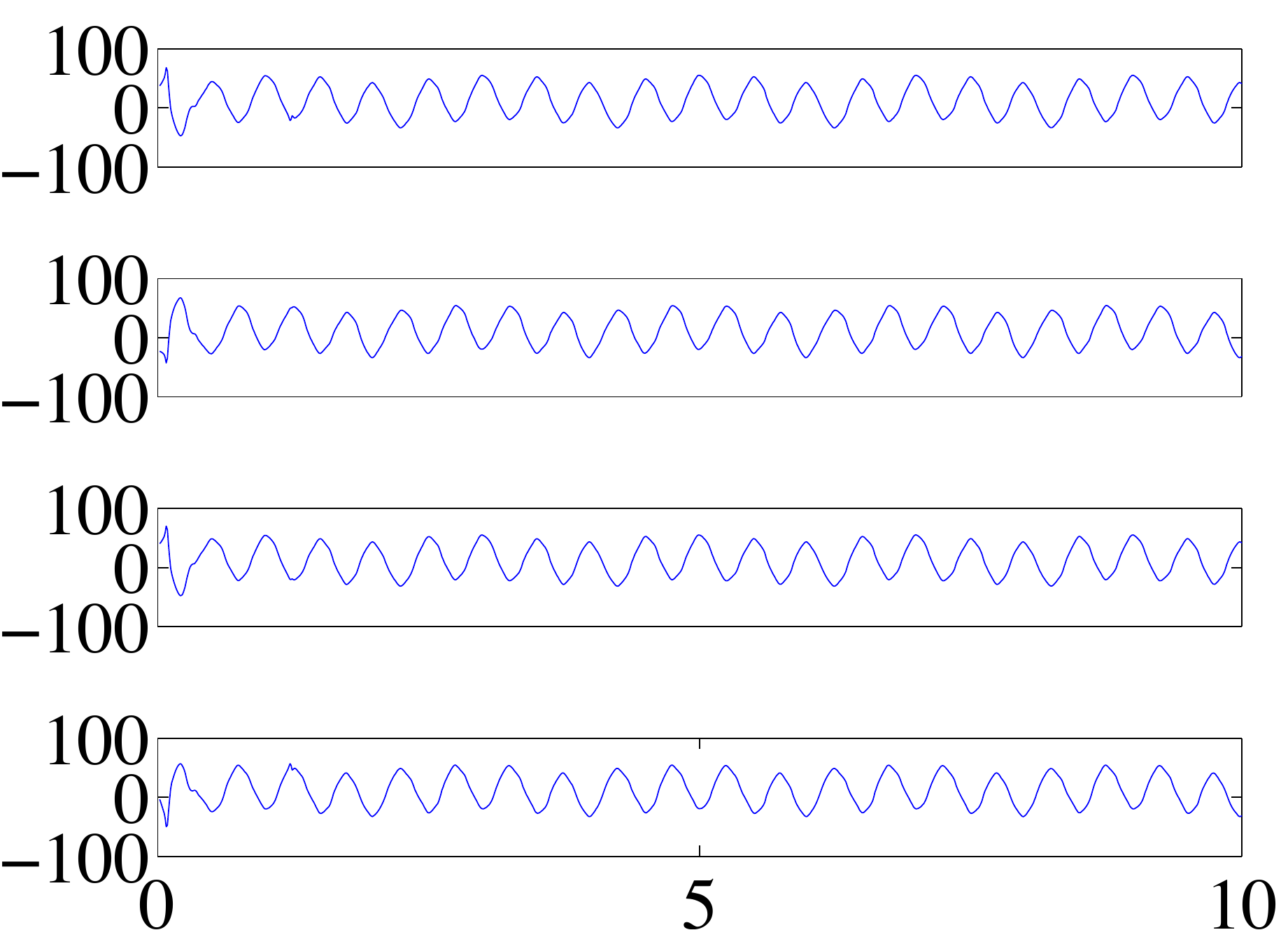}}
}
\caption{Case I: robust position controlled flight mode to follow an elliptic helix}\label{fig:I}
\end{figure}

\begin{figure}
\centerline{
	\subfigure[Attitude error function $\Psi$]{
		\includegraphics[width=0.48\columnwidth]{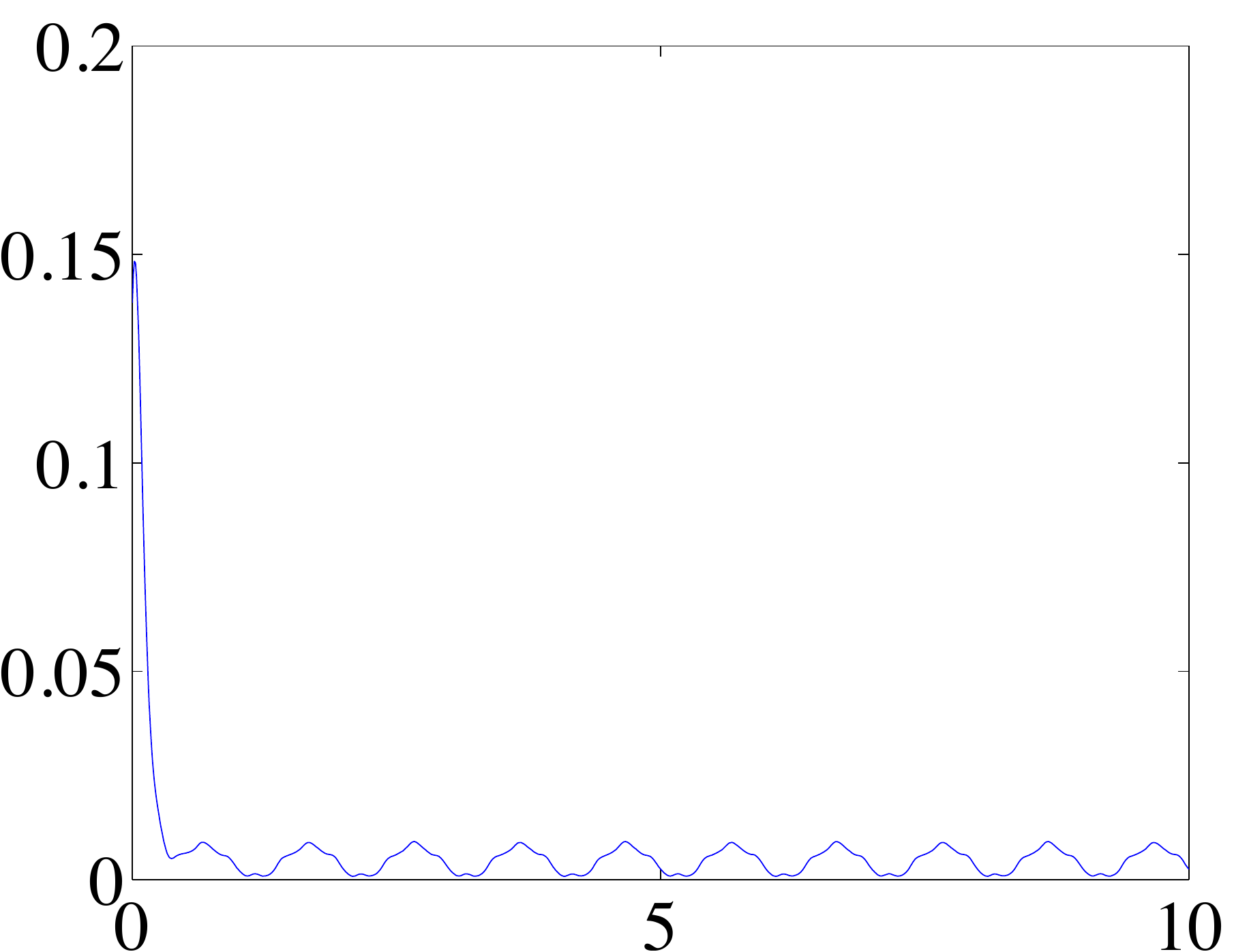}\label{fig:IbPsi}}
		\hfill
	\subfigure[Position error $e_x$ ($\mathrm{m}$)]{
		\includegraphics[width=0.485\columnwidth]{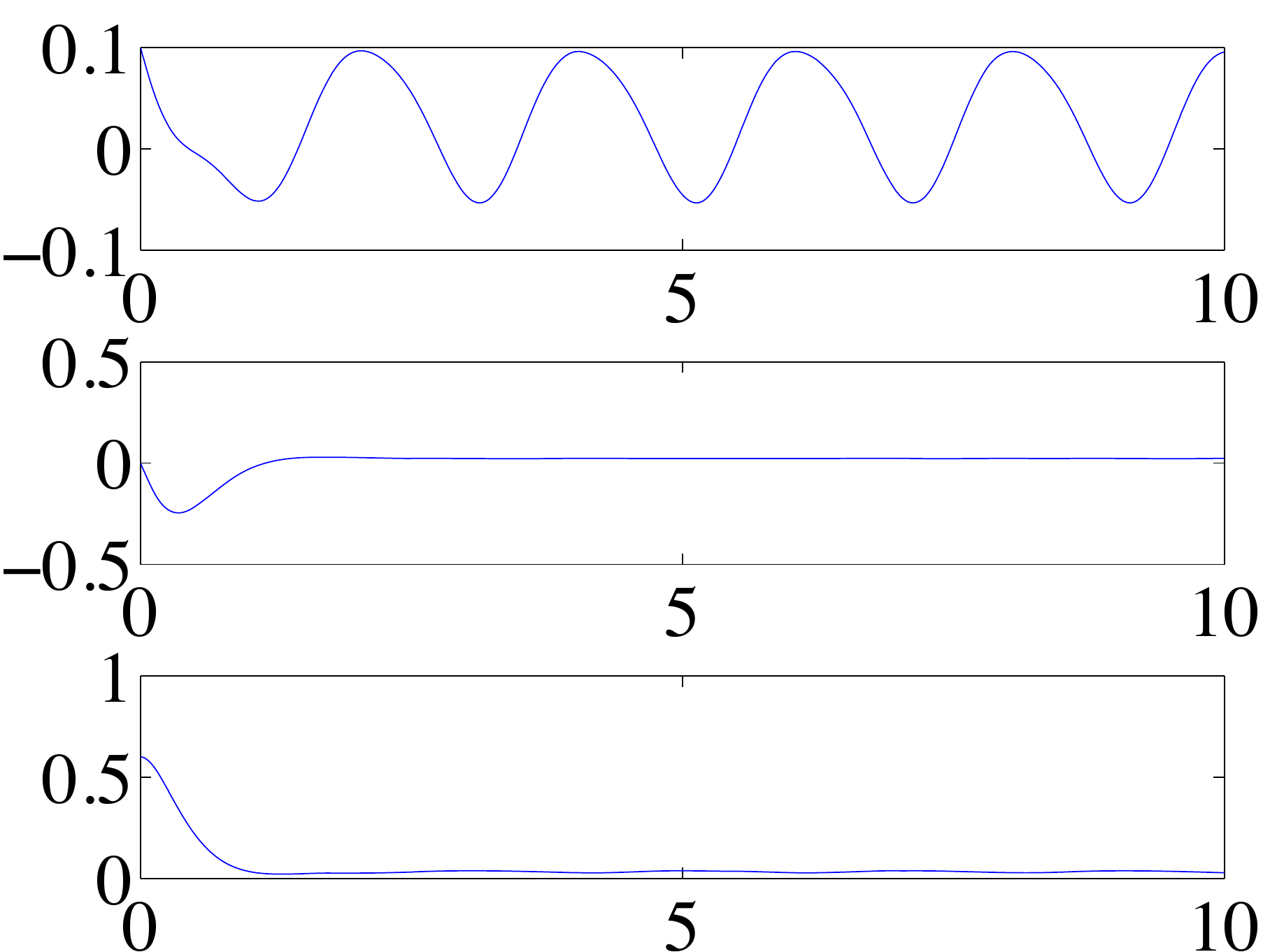}\label{fig:Ibx}}
}
\centerline{
	\subfigure[Angular velocity error $e_\Omega$ ($\mathrm{rad/sec}$)]{
		\includegraphics[width=0.48\columnwidth]{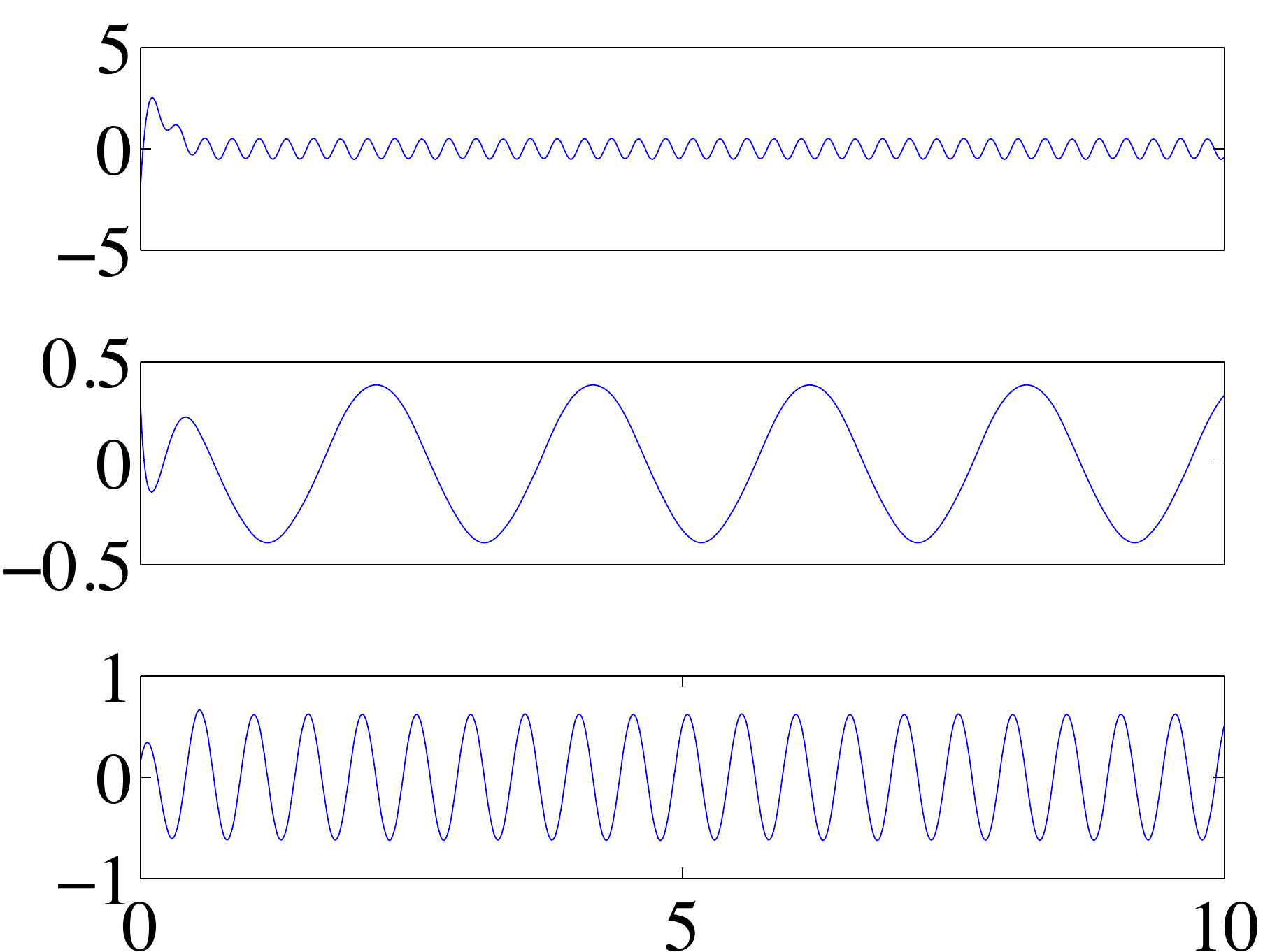}\label{fig:IbW}}
		\hfill
	\subfigure[Thrust of each rotor ($\mathrm{N}$)]{
		\includegraphics[width=0.50\columnwidth]{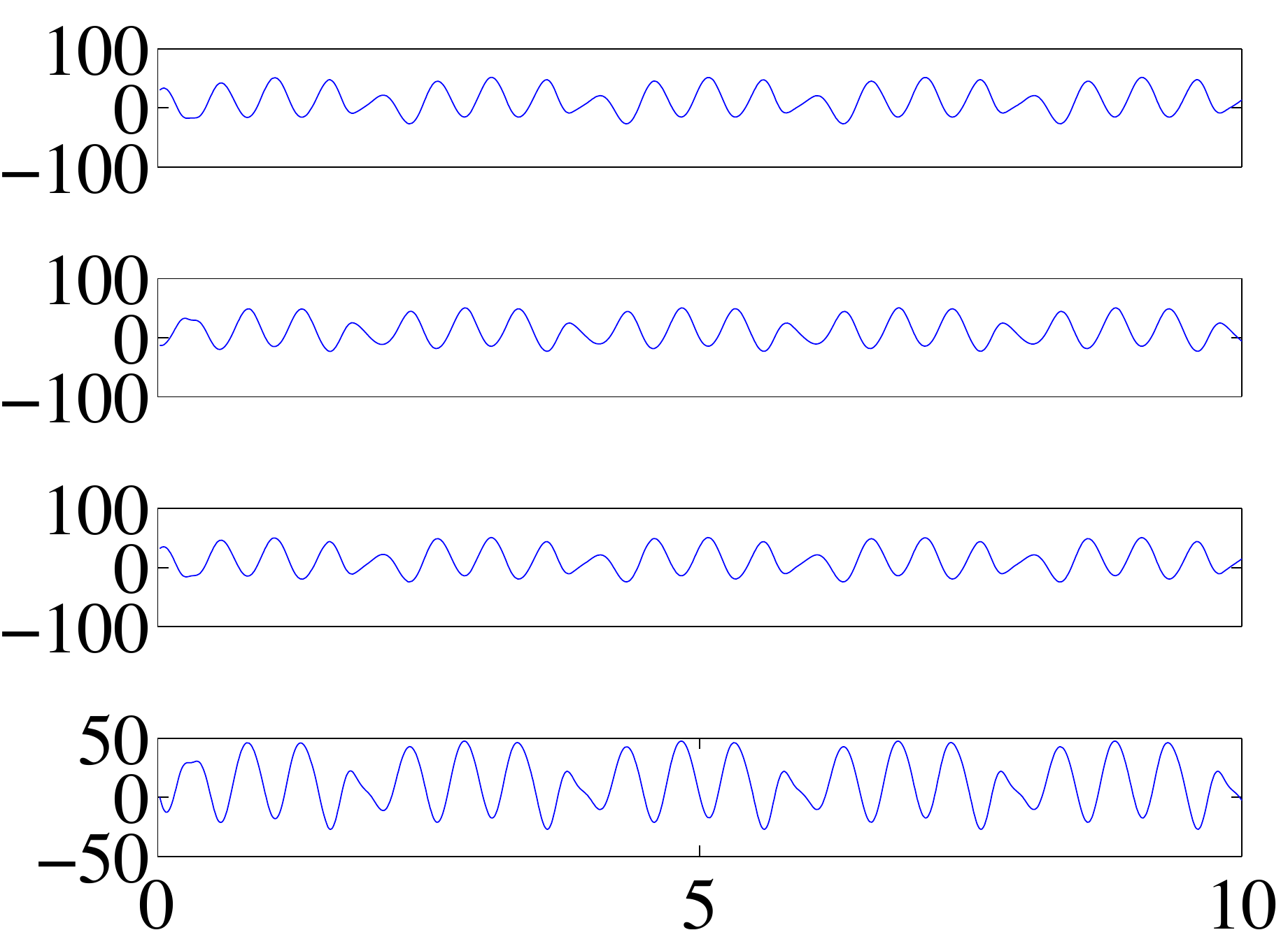}}
}
\caption{Case I: position controlled flight mode to follow an elliptic helix. The robust control input terms are set to zero, i.e. $\mu_x=\mu_R=0$, for comparison with \reffig{I}}\label{fig:Ib}
\end{figure}

\paragraph*{Case II (hovering)}

The initial conditions are given by
\begin{gather*}
x(0)=[0.1,0,0]^T\,\mathrm{m},\quad v(0)=[0,0,0]^T\,\mathrm{m/s},\\
R(0)=\exp (0.99\pi \hat e_1),\quad \Omega(0)=[0,0,0]^T\,\mathrm{rad/s},
\end{gather*}
where $e_1=[1,0,0]\in\Re^3$. The desired position command is given by
\begin{align*}
x_d(t)=[0,0,0]^T\,\mathrm{m},
\end{align*}
and the desired heading direction is fixed as $\vec b_{1_d}=[1,0,0]^T$. This describes a case that a quadrotor UAV should recover from an initially upside-down configuration.  

The initial attitude error is given by $1\leq (\Psi(0)=1.9995) < 2$, and therefore, it corresponds to Proposition \ref{prop:Pos2} that is based on both of the attitude controlled flight mode and the position controlled flight mode. 

Figure \ref{fig:II} illustrates excellent convergence properties of the proposed control system for a large initial attitude error, where the terminal position tracking error is $1.2\,\mathrm{cm}$. Figure Figure \ref{fig:IIb} shows relatively poor tracking performances with a slower convergence when there are no robust control input terms proposed in this paper. 

\begin{figure}
\centerline{
	\subfigure[Attitude error function $\Psi$]{
		\includegraphics[width=0.47\columnwidth]{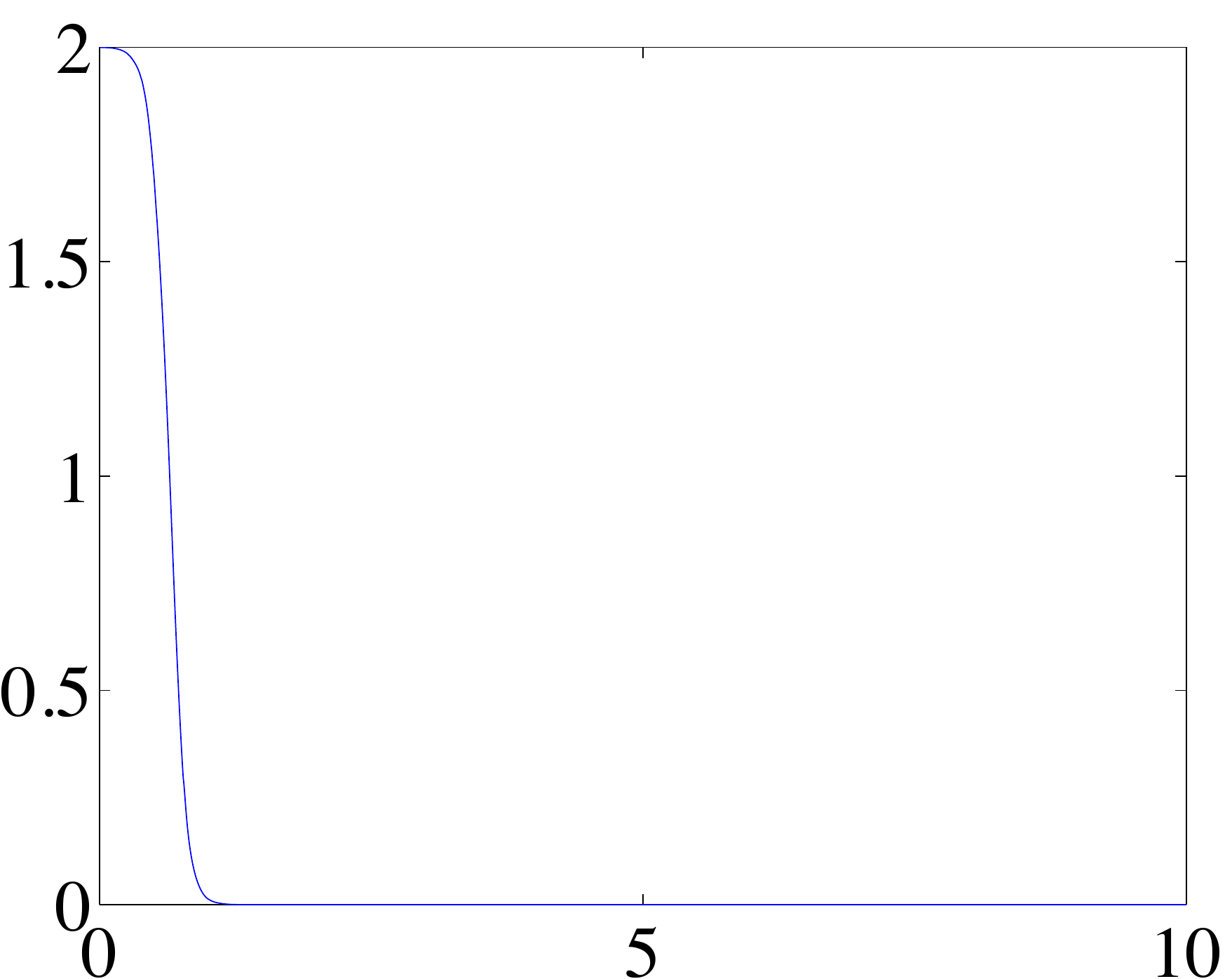}\label{fig:IPsi}}
		\hfill
	\subfigure[Position error $e_x$ ($\mathrm{m}$)]{
		\includegraphics[width=0.49\columnwidth]{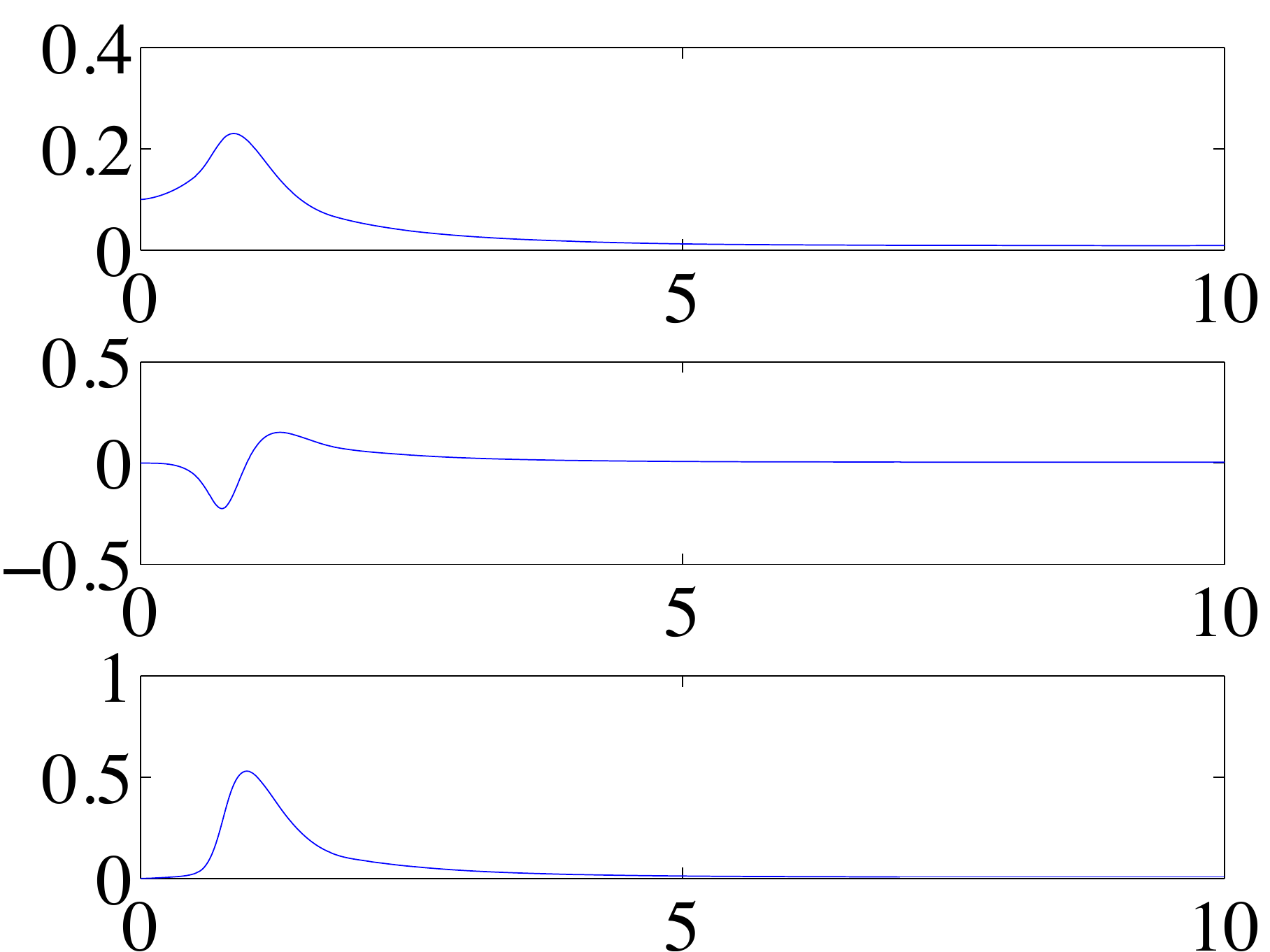}\label{fig:Ix}}
}
\centerline{
	\subfigure[Angular velocity error $e_\Omega$ ($\mathrm{rad/sec}$)]{
		\includegraphics[width=0.48\columnwidth]{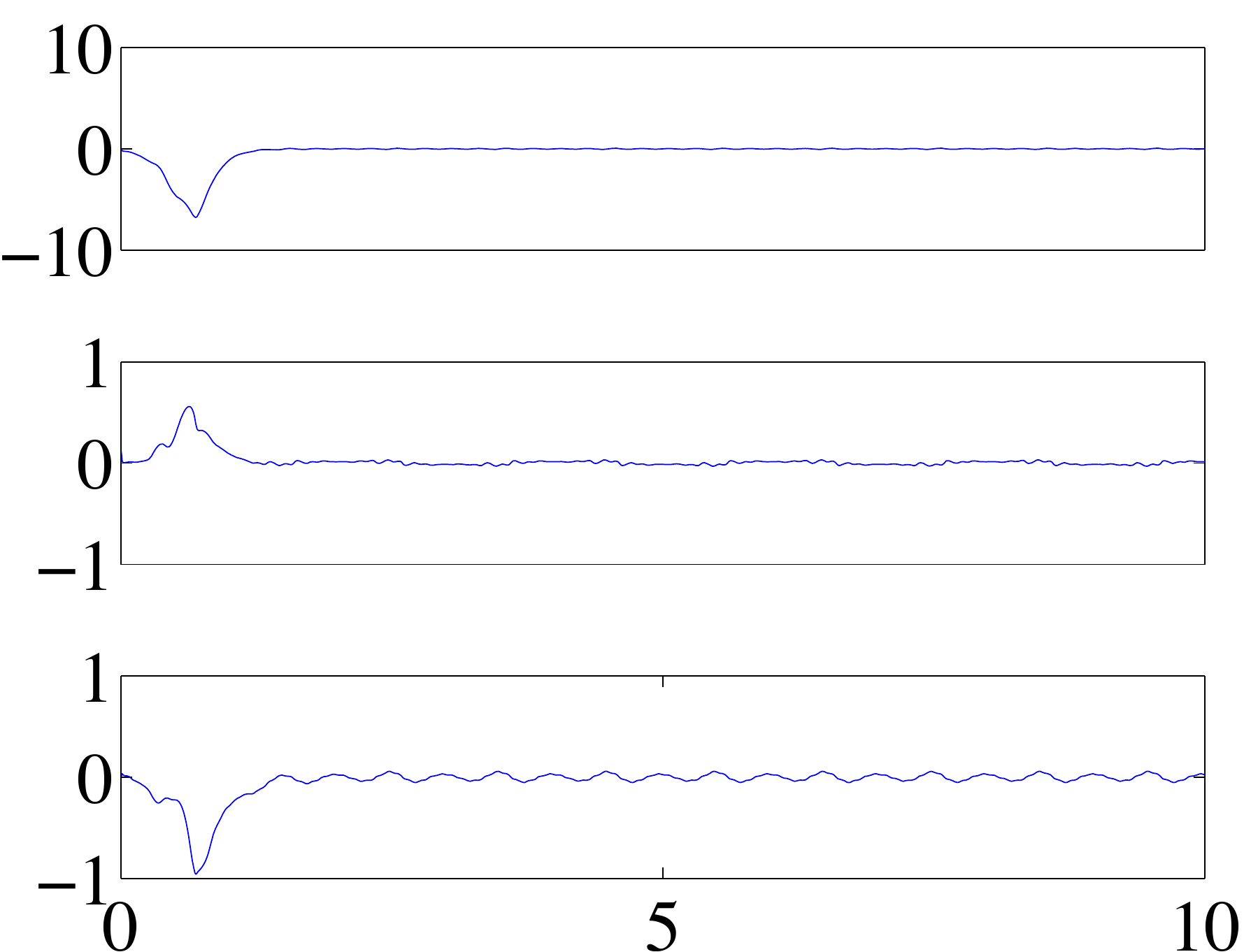}\label{fig:IW}}
		\hfill
	\subfigure[Thrust of each rotor ($\mathrm{N}$)]{
		\includegraphics[width=0.50\columnwidth]{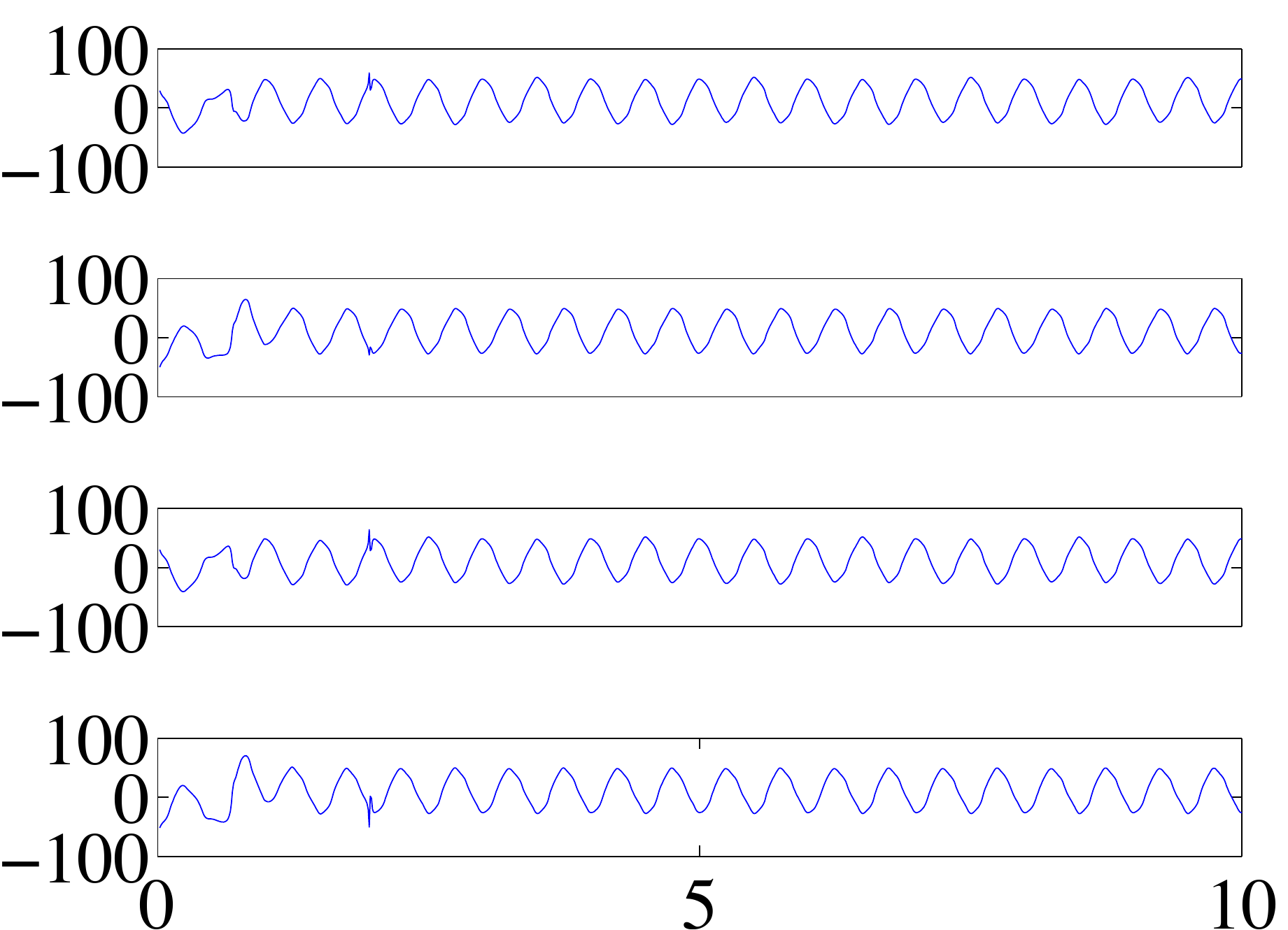}}
}
\caption{Case II: robust position controlled flight mode to recover from an initially upside-down configuration}\label{fig:II}
\end{figure}

\begin{figure}
\centerline{
	\subfigure[Attitude error function $\Psi$]{
		\includegraphics[width=0.47\columnwidth]{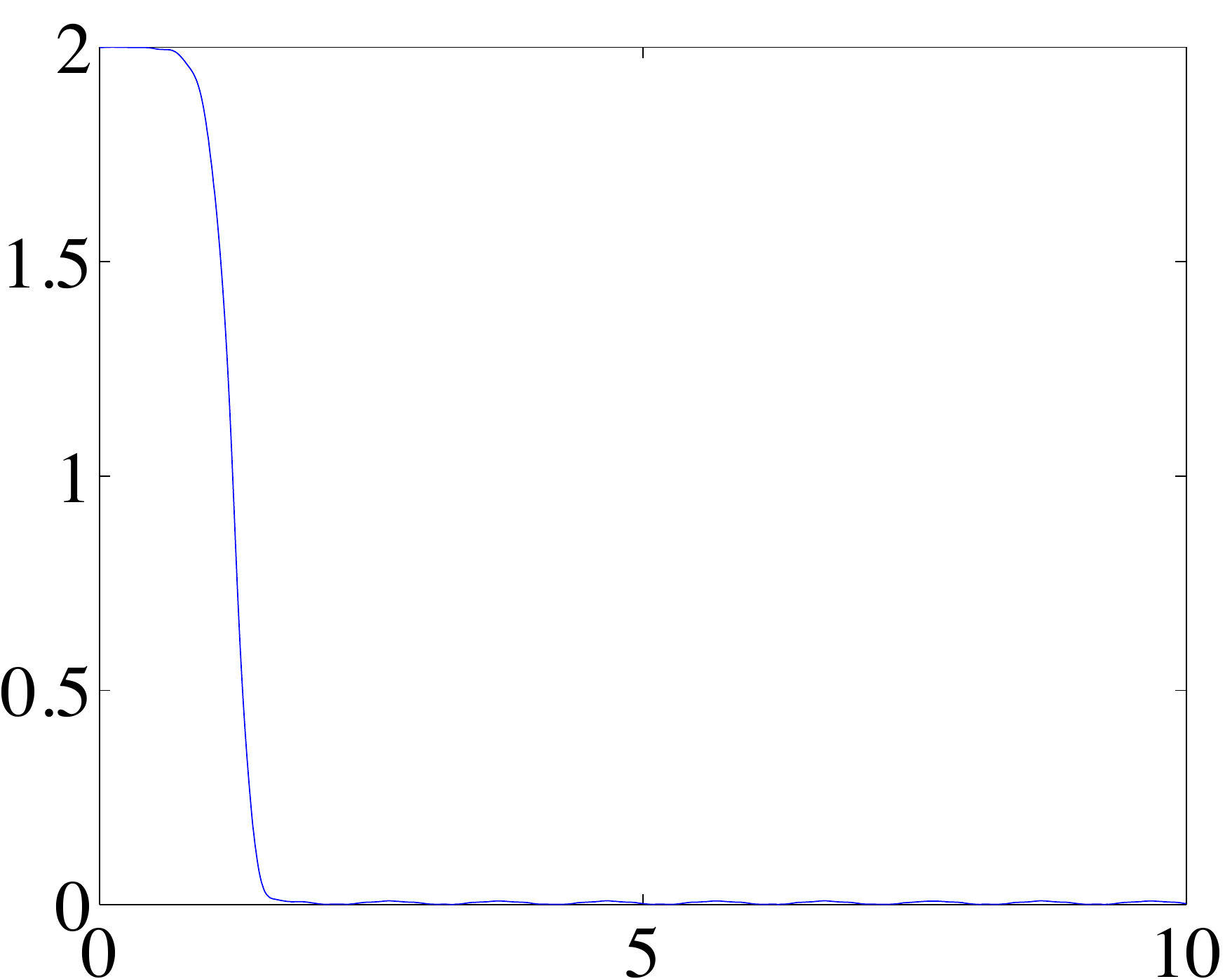}\label{fig:IbPsi}}
		\hfill
	\subfigure[Position error $e_x$ ($\mathrm{m}$)]{
		\includegraphics[width=0.49\columnwidth]{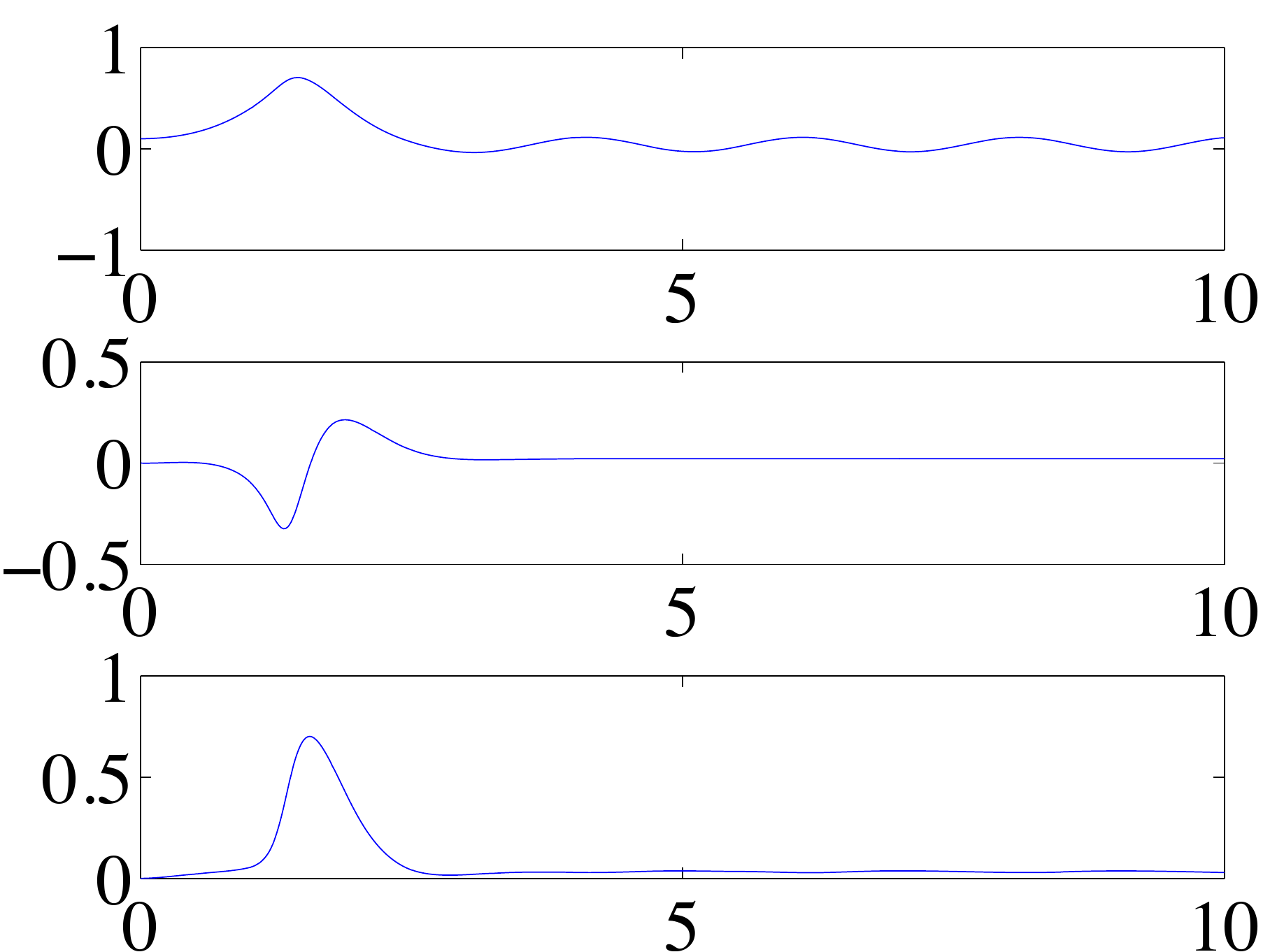}\label{fig:Ibx}}
}
\centerline{
	\subfigure[Angular velocity error $e_\Omega$ ($\mathrm{rad/sec}$)]{
		\includegraphics[width=0.48\columnwidth]{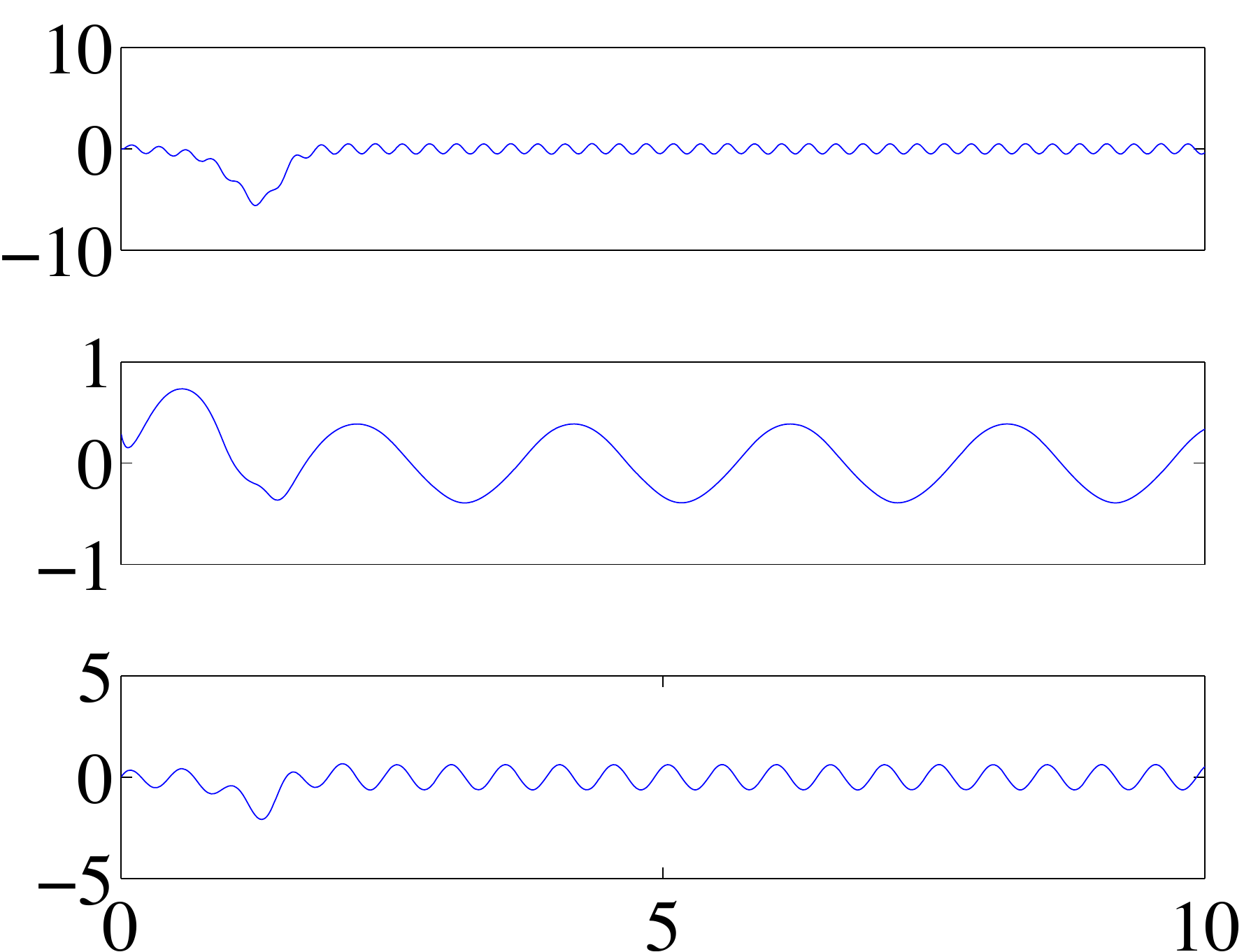}\label{fig:IbW}}
		\hfill
	\subfigure[Thrust of each rotor ($\mathrm{N}$)]{
		\includegraphics[width=0.50\columnwidth]{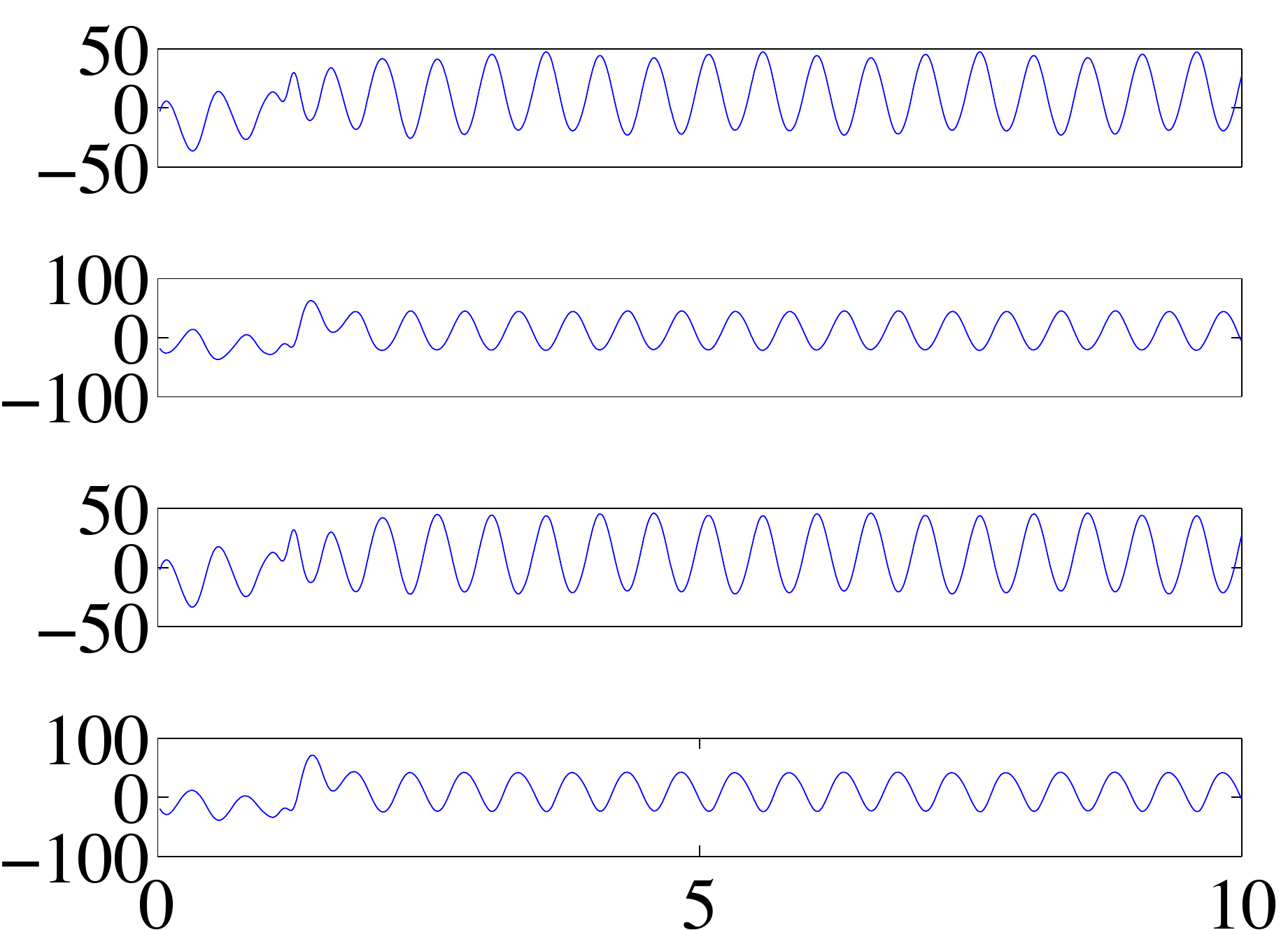}}
}
\caption{Case II: robust position controlled flight mode to recover from an initially upside-down configuration. The robust control input terms are set to zero, i.e. $\mu_x=\mu_R=0$, for comparison with \reffig{II}}\label{fig:IIb}
\end{figure}

\appendix

\section{Properties and Proofs}

\subsection{Properties of the \textit{Hat} Map}\label{app:hat}
\noindent The hat map $\hat\cdot :\Re^3\rightarrow\so$ is defined as
\begin{align}
    \hat x = \begin{bmatrix} 0 & -x_3 & x_2\\
                                x_3 & 0 & -x_1\\
                                -x_2 & x_1 & 0 \end{bmatrix}\label{eqn:hat}
\end{align}
for $x=[x_1;x_2;x_3]\in\Re^3$. This identifies the Lie algebra $\so$ with $\Re^3$ using the vector cross product in $\Re^3$. The inverse of the hat map is referred to as the \textit{vee} map, $\vee:\so\rightarrow\Re^3$. Several properties of the hat map are summarized as follows.
\begin{gather}
    \hat x y = x\times y = - y\times x = - \hat y x,\\
    -\frac{1}{2}\tr{\hat x \hat y} = x^T y,\\
    \tr{\hat x A}=\tr{A\hat x }=\frac{1}{2}\tr{\hat x (A-A^T)}=-x^T (A-A^T)^\vee,\label{eqn:hat1}\\
    \hat x  A+A^T\hat x=(\braces{\tr{A}I_{3\times 3}-A}x)^{\wedge},\label{eqn:xAAx}\\
R\hat x R^T = (Rx)^\wedge,\label{eqn:hat2}
\end{gather}
for any $x,y\in\Re^3$, $A\in\Re^{3\times 3}$, and $R\in\SO$.

\subsection{Proof of Proposition \ref{prop:Att}}\label{sec:pfAtt}

We first find the error dynamics for $e_R,e_\Omega$, and define a Lyapunov function. Then, we find conditions on control parameters to guarantee the boundedness of tracking errors.


\paragraph{Attitude Error Dynamics}

The attitude error dynamics for $\Psi,e_R,e_\Omega$ are developed in~\cite{Lee11}, and they are summarized as follows:
\begin{gather}
\frac{d}{dt}(\Psi(R,R_d))  = e_R\cdot e_\Omega,\label{eqn:Psidot}\\
\dot e_R  = E(R,R_d)e_\Omega,\label{eqn:eRdot}\\
\dot e_\Omega  = J^{-1}(-\Omega\times J\Omega + u+\Delta_R)+\hat\Omega R^T R_d\Omega_d- R^T R_d{\dot \Omega}_d,\label{eqn:eWdot00}
\end{gather}
where the matrix $E(R,R_d)\in\Re^{3\times 3}$ is given by 
\begin{align}
E(R,R_d)&=\frac{1}{2}(\trs{R^T R_d}I -R^T R_d).\label{eqn:E}
\end{align}
We can show that $\|E(R,R_d)\|\leq 1$ to obtain
\begin{align}
\norm{\dot e_R}\leq \norm{e_\Omega}.\label{eqn:neRdot}
\end{align}
Substituting the control moment \refeqn{aM} into \refeqn{eWdot00}, 
\begin{align}
J\dot e_\Omega = -k_R e_R -k_\Omega e_\Omega+\Delta_R +\mu_R.\label{eqn:eWdot}
\end{align}
In short, the attitude error dynamics are given by equations \refeqn{Psidot}, \refeqn{eRdot}, \refeqn{eWdot}, and they satisfy \refeqn{neRdot}.

\paragraph{Lyapunov Candidate}
Let a Lyapunov candidate $\mathcal{V}_2$ be 
\begin{align}
\mathcal{V}_2 = \frac{1}{2} e_\Omega \cdot J e_\Omega + k_R\, \Psi(R,R_d)+c_2 e_R\cdot e_\Omega.\label{eqn:V2}
\end{align}
We analyzes the properties of $\mathcal{V}_2$ along the solutions of the controlled system in the following domain $D_2$:
\begin{align}
D_2 = \{ (R,\Omega)\in \SO\times\Re^3\,|\, \Psi(R,R_d)<\psi_2\}.\label{eqn:D2}
\end{align}

From \refeqn{PsiLB}, \refeqn{PsiUB}, the attitude error function is bounded in $D_2$ as follows:
\begin{align}
\frac{1}{2} \norm{e_R}^2 \leq  \Psi(R,R_d) \leq \frac{1}{2-\psi_2} \norm{e_R}^2\label{eqn:eRPsi},
\end{align}
which implies that $\Psi$ is positive-definite and decrescent. It follows that the Lyapunov function $\mathcal{V}_2$ is bounded as
\begin{gather}
z_2^T M_{21} z_2 \leq \mathcal{V}_2 \leq z_2^T M_{22} z_2,
\label{eqn:V2b}
\end{gather}
where $z_2 =[\|e_R\|,\;\|e_\Omega\|]^T\in\Re^2$, and the matrices $M_{12},M_{22}$ are given by
\begin{align}
M_{21} = \frac{1}{2}\begin{bmatrix} k_R & -c_2 \\ -c_2 & \lambda_{m}(J)  \end{bmatrix},\,
M_{22} = \frac{1}{2}\begin{bmatrix} \frac{2k_R}{2-\psi_2} & c_2 \\ c_2 & \lambda_{M}(J)\end{bmatrix}.
\end{align}

From equations \refeqn{Psidot}, \refeqn{eRdot}, \refeqn{eWdot}, the time derivative of $\mathcal{V}_2$ along the solution of the controlled system is given by
\begin{align}
\dot{\mathcal{V}}_2  
& = -k_\Omega \|e_\Omega\|^2 -c_2k_R e_R\cdot J^{-1}e_R +c_2 E(R,R_d) e_\Omega \cdot e_\Omega\nonumber\\
&\quad  - c_2k_\Omega e_R\cdot J^{-1} e_\Omega+ (e_\Omega+c_2 J^{-1}e_R)\cdot(\Delta_R+\mu_R).\label{eqn:V2dot}
\end{align}
Since $\|E(R_d,R)\|\leq 1$, this is bounded by
\begin{align}
\dot{\mathcal{V}}_2   
& \leq - z_2^T W_2 z_2+ e_A \cdot (\Delta_R+\mu_R),\label{eqn:V2dot1}
\end{align}
where $e_A=e_\Omega+c_2 J^{-1}e_R\in\Re^3$ and the matrix $W_2\in\Re^{2\times 2}$ is given by
\begin{align}
W_2 = \begin{bmatrix} \frac{c_2k_R}{\lambda_{M}(J)} & -\frac{c_2k_\Omega}{2\lambda_{m}(J)} \\ 
-\frac{c_2k_\Omega}{2\lambda_{m}(J)} & k_\Omega-c_2 \end{bmatrix}.
\end{align}

Substituting \refeqn{muR}, the last term of \refeqn{V2dot1} is bounded by
\begin{align*}
e_A\cdot(\Delta_R+\mu_R) & = \delta_R\|e_A\|-\frac{\delta_R^2 \|e_A\|^2}{\delta_R \|e_A\|+\epsilon_R}\\
& = \epsilon_R\frac{\delta_R\|e_A\|}{\delta_R\|e_A\|+\epsilon_R} \leq \epsilon_R
\end{align*}
to obtain
\begin{align}
\dot{\mathcal{V}}_2 \leq - z_2^T W_2 z_2+\epsilon_R,\label{eqn:V2dot10}
\end{align}

\paragraph{Boundedness}

The condition \refeqn{c2} for the constant $c_2$ guarantees that the matrix $W_2$ in \refeqn{V2dot10} and the matrices $M_{21},M_{22}$ in \refeqn{V2b} are positive-definite. Therefore, we obtain
\begin{gather}
\lambda_{m}(M_{21})\|z_2\|^2 \leq \mathcal{V}_2 \leq \lambda_{M}(M_{22}) \|z_2\|^2,\\
\dot{\mathcal{V}}_2\leq -\lambda_{m}(W_2) \|z_2\|^2+\epsilon_R. \label{eqn:V2bb}
\end{gather}
This implies that $\dot {\mathcal{V}}_2 < 0 $ when 
\begin{align*}
\mathcal{V}_2 > \braces{\frac{\lambda_{M}(M_{22})}{\lambda_{m}(W_2)}\epsilon_R\triangleq d_1}.
\end{align*}


Consider a sub-level set of the Lyapunov function $\mathcal{V}_2$, defined as $S_\gamma = \{(R,\Omega)\in\SO\times\Re^3\,|\,\mathcal{V}_2 \leq \gamma\}$ for a positive constant $\gamma$. If $\gamma$ satisfies the following inequality
\begin{align*}
\gamma < \braces{\lambda_m(M_{21}) \psi_2(2-\psi_2)\triangleq d_2},
\end{align*}
then we can guarantee that $S_\gamma$ is a subset of the domain $D_2$ defined in \refeqn{D2}.


In short, a sub-level set of the Lyapunov function, $S_\gamma$ is a positively invariant set when $d_1<\gamma < d_2$, and any solution starting in $S_\gamma$ exponentially converges to $S_{d_1}$. To guarantee the existence of such a set, we require
\begin{align*}
\braces{d_1 = \frac{\lambda_{M}(M_{22})}{\lambda_{m}(W_2)}\epsilon_R} < \braces{\lambda_{m}(M_{21}) \psi_2(2-\psi_2)\triangleq d_2},
\end{align*}
which can be achieved by \refeqn{epsilonR}. Then, according to Theorem 5.1 in~\cite{Kha96}, the attitude tracking errors are uniformly ultimately bounded, and the corresponding ultimate bound is estimated by
\begin{align*}
S_{d_1}\subset \braces{\|z_2\|^2 \leq \frac{\lambda_{M}(M_{22})}{\lambda_{m}(M_{21})\lambda_{m}(W_2)}\epsilon_R}.
\end{align*}

\subsection{Proof of Proposition \ref{prop:Pos}}\label{sec:pfPos}
\setcounter{paragraph}{0}

We first derive the tracking error dynamics and a Lyapunov function for the translational dynamics of a quadrotor UAV, and later it is combined with the stability analyses of the rotational dynamics in Appendix \ref{sec:pfAtt} to guarantee the boundedness of tracking errors.

The subsequent analyses are developed in the domain $D_1$
\begin{align}
D_1=\{&(e_x,e_v,R,e_\Omega)\in\Re^3\times\Re^3\times \SO\times\Re^3\,|\,\nonumber\\
& \|e_x\|< e_{x_{\max}},\;\Psi< \psi_1\},\label{eqn:D}
\end{align}
Similar to \refeqn{eRPsi}, we can show that 
\begin{align}
\frac{1}{2} \norm{e_R}^2 \leq  \Psi(R,R_c) \leq \frac{1}{2-\psi_1} \norm{e_R}^2\label{eqn:eRPsi1}.
\end{align}

\paragraph{Translational Error Dynamics} The time derivative of the position error is $\dot e_x=e_v$. The time-derivative of the velocity error is given by
\begin{align}
m\dot e_v = m\ddot x -m\ddot x_d = mg e_3 - fRe_3 -m\ddot x_d+\Delta_x. \label{eqn:evdot0}
\end{align}
Consider the quantity $e_3^T R_c^T R e_3$, which represents the cosine of the angle between $b_3=Re_3$ and $b_{c_3}=R_ce_3$. Since $1-\Psi(R,R_c)$ represents the cosine of the eigen-axis rotation angle between $R_c$ and $R$, we have $1 > e_3^T R_c^T R e_3> 1-\Psi(R,R_c)>0$ in $D_1$. Therefore, the quantity $\frac{1}{e_3^T R_c^T R e_3}$ is well-defined. To rewrite the error dynamics of $e_v$ in terms of the attitude error $e_R$, we add and subtract $\frac{f}{e_3^T R_c^T R e_3}R_c e_3$ to the right hand side of \refeqn{evdot0} to obtain
\begin{align}
m\dot e_v &  = mg e_3 -m\ddot x_d- \frac{f}{e_3^T R_c^T R e_3}R_c e_3 - X+\Delta_x,\label{eqn:evdot1}
\end{align}
where $X\in\Re^3$ is defined by
\begin{align}
X=\frac{f}{e_3^T R_c^T R e_3}( (e_3^T R_c^T R e_3)R e_3 -R_ce_3).\label{eqn:X}
\end{align}
Let $A=-k_x e_x - k_v e_v -mg e_3 + m\ddot x_d+\mu_x$.
Then, from \refeqn{f}, \refeqn{Rd3}, we have $f=-A\cdot Re_3$ and ${b}_{3_c}=R_c e_3 = -A/\norm{A}$, i.e., $-A=\|A\| R_c e_3$. By combining these, we obtain $f= (\norm{A}R_c e_3)\cdot R e_3$. Therefore, the third term of the right hand side of \refeqn{evdot1} can be written as
\begin{align*}
- \frac{f}{e_3^T R_c^T R e_3}R_c e_3 & = -\frac{(\norm{A}R_c e_3)\cdot R e_3}{e_3^T R_c^T R e_3}\cdot - \frac{A}{\norm{A}}=A\\
& =-k_x e_x - k_v e_v -mg e_3 + m\ddot x_d+\mu_x.
\end{align*}
Substituting this into \refeqn{evdot1}, the error dynamics of $e_v$ can be written as
\begin{align}
m\dot e_v & =  -k_x e_x - k_v e_v - X+\Delta_x+\mu_x.\label{eqn:evdot}
\end{align}

\paragraph{Lyapunov Candidate for Translation Dynamics}
Let a Lyapunov candidate $\mathcal{V}_1$ be
\begin{align}
\mathcal{V}_1 = \frac{1}{2}k_x\|e_x\|^2  + \frac{1}{2} m \|e_v\|^2 + c_1 e_x\cdot e_v\label{eqn:V1}.
\end{align}
The derivative of ${\mathcal{V}}_1$ along the solution of \refeqn{evdot} is given by
\begin{align}
\dot{\mathcal{V}}_1 
& =  -(k_v-c_1) \|e_v\|^2 
- \frac{c_1 k_x}{m} \|e_x\|^2 
-\frac{c_1 k_v}{m} e_x\cdot e_v\nonumber\\
&\quad+\braces{X+\Delta_x+\mu_x}\cdot \braces{ \frac{c_1}{m} e_x + e_v}.\label{eqn:V1dot0}
\end{align}

From \refeqn{mux}, \refeqn{eB}, the last part of \refeqn{V1dot0} is bounded by
\begin{align}
e_B\cdot(\Delta_x + \mu_x ) & \leq \delta_x \|e_B\| -\frac{\delta_x^{\tau+2} \|e_B\|^{\tau+2}}{\delta_x^{\tau+1}\|e_B\|^{\tau+1}+\epsilon_x^{\tau+1}}\nonumber\\
& = \frac{\delta_x \|e_B\|\epsilon_x^{\tau+1}}{\delta_x^{\tau+1}\|e_B\|^{\tau+1}+\epsilon_x^{\tau+1}}\leq \epsilon_x.\label{eqn:eBD_bound}
\end{align}
The last inequality is satisfied, since if $\delta_x\|e_B\|\leq \epsilon_x$
\begin{align*}
\delta_x \|e_B\|\frac{\epsilon_x^{\tau+1}}{\delta_x^{\tau+1}\|e_B\|^{\tau+1}+\epsilon_x^{\tau+1}}\leq \delta_x\|e_B\| \leq \epsilon_x,
\end{align*}
and if $\delta_x\|e_B\| > \epsilon_x$
\begin{align*}
\frac{\delta_x^{\tau+1} \|e_B\|^{\tau+1}}{\delta_x^{\tau+1}\|e_B\|^{\tau+1}+\epsilon_x^{\tau+1}}
\frac{\epsilon_x^{\tau+1}}{(\delta_x \|e_B\|)^\tau}
\leq \parenth{\frac{\epsilon_x}{\delta_x \|e_B\|}}^\tau \epsilon_x \leq \epsilon_x.
\end{align*}

Now we find a bound of $X$ given by \refeqn{X}. Since $f=\|A\| (e_3^T R_c^T R e_3)$, we have
\begin{align*}
\norm{X} & \leq \|A\|\,\| (e_3^T R_c^T R e_3)R e_3 -R_ce_3\|\\
& \leq( k_x \|e_x\| + k_v \|e_v\| + B+\delta_x)\, \| (e_3^T R_c^T R e_3)R e_3 -R_ce_3\|.
\end{align*}
The last term $\| (e_3^T R_c^T R e_3)R e_3 -R_ce_3\|$ represents the sine of the angle between $b_3=Re_3$ and $b_{c_3}=R_c e_3$, since
\begin{align*}
(b_{3_c}\cdot b_3)b_3 - b_{3_c} = b_{3}\times (b_3\times b_{3_c}).
\end{align*}
The magnitude of the attitude error vector, $\|e_R\|$ represents the sine of the eigen-axis rotation angle between $R_c$ and $R$ (see \cite{LeeLeo}). Therefore, we have $\| (e_3^T R_c^T R e_3)R e_3 -R_ce_3\| \leq \| e_R\|$. It follows that 
\begin{align}
\| (e_3^T R_d^T R e_3)R e_3 -R_de_3\| &\leq \| e_R\| = \sqrt{\Psi(2-\Psi)}\nonumber\\
& \leq \braces{\sqrt{\psi_1 (2-\psi_1)}\triangleq\alpha}  <1.\label{eqn:eR_bound}
\end{align}
Therefore, $X$ is bounded by
\begin{align}
\norm{X} 
&\leq ( k_x \|e_x\| + k_v \|e_v\| + B+\delta_x) \|e_R\| \nonumber\\
&\leq ( k_x \|e_x\| + k_v \|e_v\| + B+\delta_x) \alpha.\label{eqn:XB}
\end{align}
Substituting \refeqn{eBD_bound}, \refeqn{XB} into \refeqn{V1dot0}, 
\begin{align}
\dot{\mathcal{V}}_1 
& \leq   -(k_v(1-\alpha)-c_1) \|e_v\|^2 
- \frac{c_1 k_x}{m}(1-\alpha) \|e_x\|^2 \nonumber\\
&\quad + \frac{c_1k_v}{m}(1+\alpha) \|e_x\|\|e_v\|\nonumber\\
&\quad +  \|e_R\| \braces{(B+\delta_x)(\frac{c_1}{m} \|e_x\| + \|e_v\|)+k_x\|e_x\|\|e_v\|}\nonumber\\
&\quad +\epsilon_x.\label{eqn:V1dot1}
\end{align}

\EditTL{In the above expression for $\dot{\mathcal{V}}_1$, there is a third-order error term, namely $k_x\|e_R\|\|e_x\|\|e_v\|$. Using \refeqn{eR_bound}, it is possible to choose its upper bound as $k_x\alpha\|e_x\|\|e_v\|$ similar to other terms, but the corresponding stability analysis becomes complicated, and the initial attitude error should be reduced further. Instead, we restrict our analysis to the domain $D_1$ defined in \refeqn{D}, and its upper bound is chosen as $k_xe_{x_{\max}}\|e_R\|\|e_v\|$.}

%
%

\paragraph{Lyapunov Candidate for the Complete System:}
Let $\mathcal{V}=\mathcal{V}_1+\mathcal{V}_2$ be the Lyapunov candidate of the complete system.
\begin{align}
\mathcal{V} & = \frac{1}{2} k_x \|e_x\|^2 + \frac{1}{2}m \|e_v\|^2 + c_1 e_x\cdot e_v\nonumber\\
&\quad + \frac{1}{2}e_\Omega \cdot Je_  \Omega + k_R\Psi(R,R_d) + c_2 e_R\cdot e_\Omega.\label{eqn:V}
\end{align}
Using \refeqn{eRPsi1}, the bound of the  Lyapunov candidate $\mathcal{V}$ can be written as
\begin{align}
z_1^T M_{11} z_1 + z_2^T M_{21} z_2 \leq \mathcal{V} \leq z_1^T M_{12} z_1 + z_2^T M'_{22} z_2,\label{eqn:Vb}
\end{align}
where $z_1=[\|e_x\|,\;\|e_v\|]^T$, $z_2=[\|e_R\|,\;\|e_\Omega\|]^T\in\Re^2$, and the matrices $M_{11},M_{12},M_{21},M_{22}$ are given by
\begin{gather*}
M_{11} = \frac{1}{2}\begin{bmatrix} k_x & -c_1 \\ -c_1 & m\end{bmatrix},\quad
M_{12} = \frac{1}{2}\begin{bmatrix} k_x & c_1 \\ c_1 & m\end{bmatrix},\\
M_{21} = \frac{1}{2}\begin{bmatrix} k_R & -c_2 \\ -c_2 & \lambda_{m}(J)  \end{bmatrix},\quad
M'_{22} = \frac{1}{2}\begin{bmatrix} \frac{2k_R}{2-\psi_1} & c_2 \\ c_2 & \lambda_{M}(J)\end{bmatrix}.
\end{gather*}

Using \refeqn{V2dot10} and \refeqn{V1dot1}, the time-derivative of $\mathcal{V}$ is given by
\begin{align}
\dot{\mathcal{V}} \leq -z_1^T W_1 z_1  + z_1^T W_{12} z_2 - z_2^T W_2 z_2+\epsilon_x+\epsilon_R, \label{eqn:Vdotb}
\end{align}
where $W_1,W_{12},W_2\in\Re^{2\times 2}$ are defined as follows:
\begin{align}
W_1 &= \begin{bmatrix} \frac{c_1k_x}{m}(1-\alpha) & -\frac{c_1k_v}{2m}(1+\alpha)\\
-\frac{c_1k_v}{2m}(1+\alpha) & k_v(1-\alpha)-c_1\end{bmatrix},\\
W_{12}&=\begin{bmatrix}
\frac{c_1}{m}(B+\delta_x) & 0 \\ B+\delta_x+k_xe_{x_{\max}} & 0\end{bmatrix},\\
W_2 &= \begin{bmatrix} \frac{c_2k_R}{\lambda_{M}(J)} & -\frac{c_2k_\Omega}{2\lambda_{m}(J)} \\ 
-\frac{c_2k_\Omega}{2\lambda_{m}(J)} & k_\Omega-c_2 \end{bmatrix}.
\end{align}

\paragraph{Boundedness} 
Under the given conditions \refeqn{c1b}, \refeqn{c2b}, all of the matrices $M_{11}$, $M_{12}$, $M_{21}$, $M'_{22}$, $W_{1}$, and $W_2$ are positive-definite. Therefore, the Lyapunov function $\mathcal{V}$ is positive-definite and decrescent to obtain
\begin{align}
\min\{ \lambda_m(M_{11}),& \lambda_m(M_{21})\} \|z\|^2 \leq\mathcal{V}\nonumber\\
& \leq \max\{ \lambda_M(M_{12}),\lambda_M(M_{22}')\}\|z\|^2,
\end{align}
where $z=[\|z_1\|,\,\|z_2\|]^T\in\Re^2$, and the time-derivative of $\mathcal{V}$ is bounded by
\begin{align}
\dot{\mathcal{V}} &\leq -\lambda_{m}(W_1)\|z_1\|^2 +\|W_{12}\|_2 \|z_1\|\|z_2\| - \lambda_{m} (W_{2})\|z_2\|^2\nonumber\\
&\quad +\epsilon_x+\epsilon_R\nonumber\\
& = -z^T W z + \epsilon_x+\epsilon_R\nonumber\\
& \leq -\lambda_m(W) \|z\|^2 + \epsilon_x+\epsilon_R.
\end{align}
where the matrix $W\in\Re^{2\times 2}$ is given by
\begin{align}
W=\begin{bmatrix}
\lambda_{m}(W_1) & -\frac{1}{2}\|W_{12}\|_2\\
-\frac{1}{2}\|W_{12}\|_2 & \lambda_m(W_2)
\end{bmatrix}.
\end{align}

Similar to the proof of Proposition \ref{prop:Att}, we can show that the tracking errors are uniformly ultimately bounded if the constants $\epsilon_x,\epsilon_R$ are sufficiently small, as given in \refeqn{epsilon_bound}, and the corresponding ultimate bound is given by \refeqn{uubPos}.

\subsection{Proof of Proposition \ref{prop:Pos2}}\label{sec:pfPos2}
The given assumptions satisfy the assumption of Proposition \ref{prop:Att}, from which the tracking error $z_2=[\|e_R\|,\|e_\Omega\|]^T$ is guaranteed to exponentially decrease until it satisfies the bound given by \refeqn{uubA}. But, \refeqn{epsilonR_bound_Pos2} guarantees that the attitude error enters the region defined by \refeqn{Psi0} in a finite time $t^*$. 

Therefore, if we show that the tracking error $z_1=[\|e_x\|,\|e_v\|]^T$ is bounded in $t\in[0, t^*]$ as well, then the complete tracking error $(z_1,z_2)$ is uniformly ultimately bounded.

The boundedness of $z_1$ is shown as follows. The error dynamics or $e_v$ can be written as
\begin{align*}
m\dot e_v = mg e_3 -fRe_3 -m\ddot x_d+\Delta_x.
\end{align*}
Let ${\mathcal{V}}_3$ be a positive-definite function of $\|e_x\|$ and $\|e_v\|$:
\begin{align*}
{\mathcal{V}}_3 = \frac{1}{2}\|e_x\|^2  + \frac{1}{2}m \|e_v\|^2.
\end{align*}
Then, we have $\|e_x\|\leq \sqrt{2\mathcal{V}_3}$, $\|e_v\|\leq \sqrt{\frac{2}{m}\mathcal{V}_3}$. The time-derivative of $\mathcal{V}_3$ is given by
\begin{align*}
\dot{\mathcal{V}}_3 & =  e_x\cdot e_v + e_v\cdot (mg e_3 -f R e_3- m\ddot x_d+\Delta_x) \\
& \leq \|e_x\|\|e_v\| + \|e_v\| (B+\delta_x) + \|e_v\| \|R e_3\| |f|.
\end{align*}
From \refeqn{f}, we obtain
\begin{align*}
\dot{\mathcal{V}}_3 & \leq \|e_x\|\|e_v\| + \|e_v\| (B+\delta_x)\\
&\quad + \|e_v\| (k_x\|e_x\|+k_v\|e_v\| + B+\delta_x)\\
& = k_v \|e_v\|^2 + (2(B+\delta_x)+(k_x+1)\|e_x\|)\|e_v\|\\
& \leq d_1 \mathcal{V}_3 + d_2 \sqrt{\mathcal{V}_3},
\end{align*}
where $d_1 = k_v \frac{2}{m} +2(k_x+1)\frac{1}{\sqrt{m}}$, $d_2 = 2(B+\delta_x) \sqrt{\frac{2}{m}}$. Suppose that $\mathcal{V}_3 \geq 1$ for a time interval $[t_a,t_b]\subset [0,t^*]$. In this time interval, we have $\sqrt{\mathcal{V}_3} \leq \mathcal{V}_3$. Therefore, 
\begin{align*}
\dot{\mathcal{V}}_3 \leq (d_1+d_2) \mathcal{V}_3 \quad \Rightarrow\quad \mathcal{V}_3(t) \leq \mathcal{V}_3(t_a) e^{(d_1+d_2)(t-t_a)}.
\end{align*}
Therefore, for any time interval in which $\mathcal{V}_3\geq 1$, $\mathcal{V}_3$ is bounded. This implies that $\mathcal{V}_3$, and therefore $z_1=[\|e_x\|,\|e_v\|]^T$, are bounded for $0\leq t\leq t^*$.

\bibliography{ACC12}

\begin{thebibliography}{10}
\providecommand{\url}[1]{#1}
\csname url@rmstyle\endcsname
\providecommand{\newblock}{\relax}
\providecommand{\bibinfo}[2]{#2}
\providecommand\BIBentrySTDinterwordspacing{\spaceskip=0pt\relax}
\providecommand\BIBentryALTinterwordstretchfactor{4}
\providecommand\BIBentryALTinterwordspacing{\spaceskip=\fontdimen2\font plus
\BIBentryALTinterwordstretchfactor\fontdimen3\font minus
  \fontdimen4\font\relax}
\providecommand\BIBforeignlanguage[2]{{%
\expandafter\ifx\csname l@#1\endcsname\relax
\typeout{** WARNING: IEEEtran.bst: No hyphenation pattern has been}%
\typeout{** loaded for the language `#1'. Using the pattern for}%
\typeout{** the default language instead.}%
\else
\language=\csname l@#1\endcsname
\fi
#2}}

\bibitem{ValBetPAGNCC06}
M.~Valenti, B.~Bethke, G.~Fiore, and J.~How, ``Indoor multi-vehicle flight
  testbed for fault detection, indoor multi-vehicle flight testbed for fault
  detection, isolation, and recovery,'' in \emph{Proceedings of the AIAA
  Guidance, Navigation and Control Conference}, 2006.

\bibitem{HofHuaAGNCC07}
G.~Hoffmann, H.~Huang, S.~Waslander, and C.~Tomlin, ``Quadrotor helicopter
  flight dynamics and control: Theory and experiment,'' in \emph{Proceedings of
  the AIAA Guidance, Navigation, and Control Conference}, 2007, {A}IAA
  2007-6461.

\bibitem{CasLozICSM05}
P.~Castillo, R.~Lozano, and A.~Dzul, ``Stabilization of a mini rotorcraft with
  four rotors,'' \emph{IEEE Control System Magazine}, pp. 45--55, 2005.

\bibitem{GilHofIJRR11}
J.~Gillula, G.~Hoffmann, H.~Huang, M.~Vitus, and C.~Tomlin, ``Applications of
  hybrid reachability analysis to robotic aerial vehicles,'' \emph{The
  International Journal of Robotics Research}, vol.~30, no.~3, pp. 335--354,
  2011.

\bibitem{BouSiePIICRA05}
S.~Bouabdalla and R.~Siegward, ``Backstepping and sliding-mode techniques
  applied to an indoor micro quadrotor,'' in \emph{Proceedings of the IEEE
  International Conference on Robotics and Automation}, 2005, pp. 2259--2264.

\bibitem{EfePMCCA07}
M.~Efe, ``Robust low altitude behavior control of a quadrotor rotorcraft
  through sliding modes,'' in \emph{Proceedings of the Mediterranean Conference
  on Control and Automation}, 2007.

\bibitem{RafOrtA10}
G.~Raffo, M.~Ortega, and F.~Rubio, ``An integral predictive/nonlinear
  {$H_{\infty}$} control structure for a quadrotor helicopter,''
  \emph{Automatica}, vol.~46, pp. 29--30, 2010.

\bibitem{NicMacPCCECE08}
C.~Nicol, C.~Macnab, and A.~Ramirez-Serrano, ``Robust neural network control of
  a quadrotor helicopter,'' in \emph{Proceedings of the Canadian Conference on
  Electrical and Computer Engineering}, 2008, pp. 1233--1237.

\bibitem{TayMcGITCSTI06}
A.~Tayebi and S.~McGilvray, ``Attitude stabilization of a {V}{T}{O}{L}
  quadrotor aircraft,'' \emph{IEEE Transactions on Control System Technology},
  vol.~14, no.~3, pp. 562--571, 2006.

\bibitem{MaySanITAC11}
C.~Mayhew, R.~Sanfelice, and A.~Teel, ``Quaternion-based hybrid control for
  robust global attitude tracking,'' \emph{IEEE Transactions on Automatic
  Control}, 2011.

\bibitem{MaySanPACC11b}
------, ``On the non-robustness of inconsistent quaternion-based attitude
  control systems using memoryless path-lifting schemes,'' in \emph{Proceeding
  of the American Control Conference}, 2011.

\bibitem{BhaBerSCL00}
S.~Bhat and D.~Bernstein, ``A topological obstruction to continuous global
  stabilization of rotational motion and the unwinding phenomenon,''
  \emph{Systems and Control Letters}, vol.~39, no.~1, pp. 66--73, 2000.

\bibitem{MaySanPACC11}
C.~Mayhew, R.~Sanfelice, and A.~Teel, ``On quaternion-based attitude control
  and the unwinding phenomenon,'' in \emph{Proceeding of the American Control
  Conference}, 2011.

\bibitem{BulLew05}
F.~Bullo and A.~Lewis, \emph{Geometric control of mechanical systems}, ser.
  Texts in Applied Mathematics.\hskip 1em plus 0.5em minus 0.4em\relax New
  York: Springer-Verlag, 2005, vol.~49, modeling, analysis, and design for
  simple mechanical control systems.

\bibitem{MaiBerITAC06}
D.~Maithripala, J.~Berg, and W.~Dayawansa, ``Almost global tracking of simple
  mechanical systems on a general class of {L}ie groups,'' \emph{IEEE
  Transactions on Automatic Control}, vol.~51, no.~1, pp. 216--225, 2006.

\bibitem{CabCunPICDC08}
D.~Cabecinhas, R.~Cunha, and C.~Silvestre, ``Output-feedback control for almost
  global stabilization of fully-acuated rigid bodies,'' in \emph{Proceedings of
  IEEE Conference on Decision and Control}, 3583-3588, Ed., 2008.

\bibitem{ChaSanICSM11}
N.~Chaturvedi, A.~Sanyal, and N.~McClamroch, ``Rigid-body attitude control,''
  \emph{IEEE Control Systems Magazine}, vol.~31, no.~3, pp. 30--51, 2011.

\bibitem{LeeLeo}
\BIBentryALTinterwordspacing
T.~Lee, M.~Leok, and N.~McClamroch, ``Control of complex maneuvers for a
  quadrotor {UAV} using geometric methods on {SE(3)},'' arXiv. [Online].
  Available: \url{http://arxiv.org/abs/1003.2005}
\BIBentrySTDinterwordspacing

\bibitem{ChaMcCITAC09}
N.~Chaturvedi, N.~H. McClamroch, and D.~Bernstein, ``Asymptotic smooth
  stabilization of the inverted 3-{D} pendulum,'' \emph{IEEE Transactions on
  Automatic Control}, vol.~54, no.~6, pp. 1204--1215, 2009.

\bibitem{Lee11}
\BIBentryALTinterwordspacing
T.~Lee, ``Robust adaptive geometric tracking controls on {SO(3)} with an
  application to the attitude dynamics of a quadrotor {UAV},'' arXiv, 2011.
  [Online]. Available: \url{http://arxiv.org/abs/1108.6031}
\BIBentrySTDinterwordspacing

\bibitem{PouMahACRA06}
P.~Pounds, R.~Mahony, and P.~Corke, ``Modeling and control of a large quadrotor
  robot,'' \emph{Control Engineering Practice}, vol.~18, pp. 691--699, 2010.

\bibitem{Kha96}
H.~Khalil, \emph{Nonlinear Systems}, 2nd Edition, Ed.\hskip 1em plus 0.5em
  minus 0.4em\relax Prentice Hall, 1996.

\end{thebibliography}
\bibliographystyle{IEEEtran}

\end{document}